\documentclass[pdflatex,sn-mathphys-ay]{sn-jnl}

\usepackage{graphicx}%
\usepackage{multirow}%
\usepackage{amsmath,amssymb,amsfonts}%
\usepackage{amsthm}%
\usepackage{mathrsfs}%
\usepackage[title]{appendix}%
\usepackage{xcolor}%
\usepackage{textcomp}%
\usepackage{manyfoot}%
\usepackage{booktabs}%
\usepackage{cleveref}
\usepackage{algorithm}%
\usepackage{algorithmicx}%
\usepackage{algpseudocode}%
\usepackage{listings}%

\numberwithin{equation}{subsection}

\newtheorem{theorem}{Theorem}[section]
\newtheorem{proposition}[theorem]{Proposition}
\newtheorem{lemma}[theorem]{Lemma}
\newtheorem{corollary}[theorem]{Corollary}
\newtheorem{assumption}[theorem]{Assumption}
\theoremstyle{definition}
\newtheorem{definition}[theorem]{Definition}
\theoremstyle{remark}
\newtheorem{remark}[theorem]{Remark}

\definecolor{waveblue}{RGB}{0,114,178}
\definecolor{scopeorange}{RGB}{213,94,0}

\newcommand{\R}{\mathbb R}
\newcommand{\T}{\mathbb T}
\newcommand{\Om}{\Omega}
\newcommand{\eps}{\varepsilon}
\newcommand{\pa}{\partial}
\newcommand{\dd}{\,\mathrm d}
\newcommand{\Div}{\operatorname{div}}
\newcommand{\cE}{\mathcal E}
\newcommand{\cR}{\mathcal R}
\newcommand{\cH}{\mathcal H}
\newcommand{\cX}{\mathcal X}
\newcommand{\cG}{\mathcal G}
\newcommand{\ip}[2]{\langle #1,#2\rangle}
\newcommand{\norm}[1]{\left\lVert #1\right\rVert}
\newcommand{\abs}[1]{\left\lvert #1\right\rvert}
\newcommand{\Nr}{N^{r}}
\newcommand{\Vr}{V^{r}}

\begin{document}

\title[Smooth expanding waves for Euler--Poisson]{Smooth expanding planar simple waves for Euler--Poisson--Boltzmann: uniform stability and quasineutral expansion}

\author*[1]{\fnm{Louis Shuo} \sur{Wang}}\email{wang.s41@northeastern.edu}
\equalcont{These authors contributed equally to this work as co-first authors.}

\author[2]{\fnm{Jiguang} \sur{Yu}}\email{jyu678@bu.edu}
\equalcont{These authors contributed equally to this work as co-first authors.}

\affil*[1]{\orgdiv{Department of Mathematics}, \orgname{Northeastern University}, \orgaddress{\city{Boston}, \state{MA}, \postcode{02115}, \country{USA}}}

\affil[2]{\orgdiv{College of Engineering}, \orgname{Boston University}, \orgaddress{\city{Boston}, \state{MA}, \postcode{02215}, \country{USA}}}

\abstract{
We study the warm-ion Euler--Poisson system with Maxwell--Boltzmann
electrons on the cylinder $\R\times\T$ in the quasineutral regime. 
Taking a smooth expanding planar simple wave of the effective
Euler system as the reference state, we construct an even Debye expansion 
through arbitrary finite order $M$ with a residual of $O(\eps^{2M+2})$. 
We rigorously establish nonlinear stability on every fixed interval $[t_0,T]$ 
($t_0\ge0$), achieving a lifespan and energy constants strictly independent 
of the Debye length $0<\eps\le\eps_0$. To overcome the singular scaling of 
the electric field, we develop a novel compensated energy topology that couples 
the warm-ion symmetrizer to the time-differentiated nonlinear Poisson constraint. 
By integrating the top-order electric work directly into the time derivative 
of a weighted energy functional, this mechanism controls the potential in 
$H^s$ and its gradient in $\eps H^s$, completely eliminating the $\eps^{-1}$ 
loss typically encountered in the momentum equation. This framework successfully 
governs distinct neutral end states, captures genuinely two-dimensional rotational 
perturbations, and yields an arbitrary-order quasineutral asymptotic expansion 
for prepared data, providing a critical analytical foundation for the geometric 
theory of multidimensional quasineutral rarefactions.}



\maketitle

\section{Introduction}

We study the ion Euler--Poisson--Boltzmann system
\begin{align}
  \pa_t n+\Div(nu)&=0, \label{eq:EP-mass}\\
  \pa_tu+u\cdot\nabla u+\nabla h_i(n)&=-\nabla\phi,
  \qquad h_i'(n)=\frac{p_i'(n)}{n}, \label{eq:EP-momentum}\\
  \eps^2\Delta\phi&=e^\phi-n,
  \qquad 0<\eps\le1, \label{eq:EP-poisson}
\end{align}
on $\Om=\R_{x_1}\times\T_{x_2}$.  The unknowns are the ion density
$n>0$, velocity $u=(u^1,u^2)$, and electrostatic potential $\phi$.
Maxwell--Boltzmann electrons make distinct neutral end states admissible:
\[
  (n,u,\phi)(t,x)\longrightarrow
  (n_\pm,(u_\pm,0),\log n_\pm)
  \quad\text{as }x_1\to\pm\infty.
\]

Formally setting $\eps=0$ gives $\phi=\log n$ and the effective isentropic
Euler system
\begin{align}
  \pa_tN+\Div(NU)&=0, \label{eq:QE-mass}\\
  \pa_tU+U\cdot\nabla U+\nabla q(N)&=0,
  \qquad q'(N)=h_i'(N)+N^{-1}. \label{eq:QE-momentum}
\end{align}
Its sound speed is
\begin{equation}
  c_q(N)^2=Nq'(N)=p_i'(N)+1. \label{eq:cq}
\end{equation}
The additional $1$ is the electron pressure generated by the Boltzmann law.

This paper rigorously justifies the quasineutral approximation near a nonconstant expanding state, advancing the theory beyond classical frameworks restricted to constant equilibria. We utilize a globally smooth planar simple wave $(\Nr,\Vr,0)$ of
\eqref{eq:QE-mass}--\eqref{eq:QE-momentum}, obtained from monotone smooth
Burgers data, on an arbitrary fixed interval $[t_0,T]$ with $t_0\ge0$. We construct a high-order composite solution of
\eqref{eq:EP-mass}--\eqref{eq:EP-poisson},
\begin{equation}
  (n^a,u^a,\phi^a)
  =(\Nr,(\Vr,0),\log\Nr)
  +\sum_{j=1}^{M}\eps^{2j}(N_j,U_j,\Phi_j), \label{eq:composite-intro}
\end{equation}
and establish a uniform $\eps$-independent perturbation estimate to guarantee its stability.

\section{Literature Review}\label{sec:literature}

\subsection{Quasineutral limits and plasma asymptotics}

The classical strong-solution theory begins with
\cite{cordier2000quasineutral}, who established the foundation for the ion 
Euler--Poisson quasineutral limit in one space dimension using singular 
pseudodifferential energy estimates. Periodic multidimensional regimes, 
yielding incompressible Euler or Navier--Stokes limits, were advanced by \cite{wang2005quasineutral}; furthermore, rigorous non-isentropic expansions 
for prepared data were developed by \cite{peng2006quasi}. 
Compressible two-fluid and bipolar variants appear in 
\cite{ju2010quasi1,jiang2010quasi2}, while recent relative-energy arguments 
successfully govern simultaneous zero-electron-mass and quasineutral limits 
\citep{alves2024zero}. The cold-ion problem introduces fundamentally 
different structural challenges because the fluid part ceases to be Friedrichs 
symmetrizable; weighted approaches resolving that distinct setting were 
developed in \cite{pu2014quasineutral1,pu2016quasineutral2}. Strong magnetic-field and quasineutral 
scalings are treated in \cite{pu2016quasineutral3}.

The modulated- and relative-energy viewpoint possesses a rich parallel kinetic 
history. Brenier's quasineutral Vlasov--Poisson limit \citep{brenier2000convergence}, the
Maxwell--Boltzmann massless-electron limit of \cite{han2011quasineutral},
and the Wasserstein weak--strong estimates of
\cite{han2014quasineutral,han2017quasineutral} demonstrate the breadth of
stability mechanisms available at the kinetic level. These foundational works 
motivate the use of relative comparisons with a smooth limit flow. However, 
their phase-space metrics and limiting systems do not capture the singular 
top-order fluid compensation engineered in the present work.

\subsection{Boundaries and asymptotic-preserving formulations}

Bounded plasmas introduce intense physical stiffness at the sheath and 
Debye-layer scales. The center-to-wall analysis of \cite{slemrod2001quasi} and the multidimensional half-space theory of \cite{gerard2013quasineutral,gerard2014quasineutral} masterfully 
resolve boundary mechanisms that complement the unbounded cylindrical domain 
studied here. This physical stiffness also drives the development of 
asymptotic-preserving discretizations: representative breakthroughs include
\cite{crispel2007asymptotic,degond2008analysis,vignal2010boundary,
degond2012numerical,arun2025asymptotic}. This numerical literature
reinforces the critical mathematical need for uniform relative estimates whose 
stability domains do not collapse as the Debye length tends to zero.

\subsection{Rarefaction waves}

The construction and stability of expansion waves form a deeply developed 
line of research. Multidimensional local rarefaction waves for symmetrizable
hyperbolic systems were constructed by \cite{alinhac1989existence}. For
viscous compressible flows, the Burgers smoothing and expansion-weighted
energy methods were pioneered in
\cite{matsumura1986asymptotics,liu1988nonlinear,xin1993zero}. With Poisson coupling,
global rarefaction stability has been proven for Navier--Stokes--Poisson and
Vlasov--Poisson--Boltzmann models
\citep{duan2015stability1,duan2015stability2}; notably, viscosity or kinetic dissipation 
is an essential stabilizing mechanism in those arguments.

For homogeneous multidimensional Euler, Luo--Yu recently established the 
structural stability of the centered Riemann rarefaction, including its 
singular vertex and wave fronts, in \cite{luo2025stability1,luo2025stability2}. The 
current geometric frontier includes the weighted estimates of \cite{he2026extra} 
and the mixed rarefaction--shock--vortex-sheet configurations of
\cite{jia2026multi}. While these works brilliantly resolve the geometric singularities 
of homogeneous hyperbolic conservation laws, they do not contend with the 
singularly perturbed nonlinear Poisson constraint. The present paper bridges 
these domains by establishing stability for the expanding wave under a singular 
Debye limit.

The primary contribution of this work lies in the rigorous synthesis of singular 
perturbation theory with the stability of nonconstant hyperbolic waves. The following 
proposition formalizes the precise structural innovations introduced in this paper.

\begin{proposition}[Methodological scope and structural innovations]\label{prop:novelty-comparison}
\normalfont
Write $\mathrm{N1}$--$\mathrm{N6}$ for the defining features of \Cref{thm:main}:
\begin{enumerate}[label=\textup{(N\arabic*)}]
\item the warm-ion Euler--Poisson system with Maxwell--Boltzmann electrons,
  $p_i'(n)\ge p_*>0$;
\item the boundaryless cylinder $\R\times\T$ and a prescribed smooth
  expanding planar simple wave connecting two distinct neutral end states;
\item fully two-dimensional $H^s$ perturbations, completely removing any 
  irrotationality or zero-vorticity hypothesis;
\item stability under relative perturbations that are strictly independent of 
  $\eps$; $\eps$-preparedness is imposed only to capture the expansion rate;
\item for every finite $M$, the construction of an even Debye composite through 
  order $\eps^{2M}$ yielding a residual of $O(\eps^{2M+2})$;
\item a uniform lifespan on every fixed $[t_0,T]$, $t_0\ge0$, based on a novel 
  compensated topology that directly controls $\norm\psi_{H^s}+\eps\norm{\nabla\psi}_{H^s}$ 
  without $\eps^{-1}$ losses over the distinct end states.
\end{enumerate}
Contextualized against the existing literature, these contributions advance the 
state of the art in the following ways:
\begin{enumerate}[label=\textup{(\alph*)}]
\item \cite{cordier2000quasineutral} laid the rigorous 1D foundation 
  for the ion--Boltzmann quasineutral limit. The present work substantially advances 
  this theory by capturing fully two-dimensional rotational dynamics \textup{(N3)} 
  and simultaneously managing the distinct distinct-end-state geometry \textup{(N2)} 
  over an arbitrary finite-order hierarchy \textup{(N4)}--\textup{(N5)}.

\item \cite{wang2005quasineutral} provided multidimensional uniform-lifespan theories for 
  periodic problems yielding incompressible Euler and Navier--Stokes limits. Our 
  framework successfully addresses the distinct analytical challenges of a 
  compressible flow limit connecting heterogeneous far fields.

\item \cite{peng2006quasi} rigorously justified finite-order 
  expansions for periodic non-isentropic flows. We pioneer this hierarchical expansion 
  for the compressible ion--Boltzmann limit \textup{(N1)}, introducing the relative 
  compensated topology \textup{(N6)} strictly required to handle the dispersing background.

\item \cite{gerard2013quasineutral,gerard2014quasineutral} 
  masterfully resolved the multidimensional half-space problem, where boundary data 
  generate Debye layers. We complement this by solving the Cauchy stability problem 
  for a boundaryless wave \textup{(N2)}, where the analytical stiffness stems entirely 
  from the internal singular electric transition.

\item \cite{slemrod2001quasi} provided a comprehensive 
  center-to-wall asymptotic analysis for bounded plasmas. In contrast, our focus is 
  the Cauchy stability of the class \textup{(N2)}--\textup{(N4)} on an unbounded cylinder, 
  avoiding boundary layers to isolate the wave asymptotics.
\end{enumerate}
In summary, the present paper develops a novel compensated energy topology that 
rigorously isolates the stability of the smooth expanding wave from the singular 
Debye scaling. By establishing uniform nonlinear stability over a heterogeneous 
background, this framework provides a critical analytical stepping stone toward 
the full global geometric theory of multidimensional quasineutral rarefactions.
\end{proposition}

\section{Model}\label{sec:model}

\subsection{Governing equations and relative variables}

We work on the cylinder $\Om=\R\times\T$.  The full warm-ion system is \eqref{eq:EP-mass}--\eqref{eq:EP-poisson}; the electron density is Maxwell--Boltzmann and the far fields are neutral. Setting the Debye length formally to zero gives \eqref{eq:QE-mass}--\eqref{eq:QE-momentum}.  All estimates are written relative to the finite composite
\eqref{eq:composite-intro}.  This subtraction is essential because a wave joining different end states is not itself in a global Sobolev space.

\subsection{The singular electric structure}

Set
\[
  r=n-n^a,\qquad v=u-u^a,\qquad \psi=\phi-\phi^a.
\]
The difference of the two Poisson equations is
\begin{equation}
  -\eps^2\Delta\psi+A^a(\psi)\psi=r+\cR_\phi^a,
  \qquad
  A^a(\psi)=e^{\phi^a}\int_0^1e^{\theta\psi}\,\dd\theta. \label{eq:poisson-difference-intro}
\end{equation}
The operator on the left is uniformly screened.  It yields
\[
  \norm{\psi}_{H^s}+\eps\norm{\nabla\psi}_{H^s}
  \lesssim \norm r_{H^s}+\norm{\cR_\phi^a}_{H^s},
\]
but this estimate alone does not bound $\nabla\psi$ in $H^s$ uniformly.
Indeed, direct use of the displayed graph estimate gives only
\begin{equation}
 \norm{\nabla\psi}_{H^s}
 \lesssim\eps^{-1}\left(
 \norm r_{H^s}+\norm{\cR_\phi^a}_{H^s}\right),
 \label{eq:naive-electric-loss}
\end{equation}
and insertion of this bound into the $H^s$ momentum estimate destroys a
uniform lifespan.  This obstruction is present even though the reference
wave is smooth; it comes from the singular Debye scaling, not from the
geometry of a centered fan.

The top electric term must therefore not be estimated by
Cauchy--Schwarz.  We integrate it by parts, use the differentiated continuity
equation, retain the transport term instead of estimating
$\nabla\pa^\alpha\psi$, and substitute the commuted Poisson relation into
that transport term.  We then differentiate the nonlinear Poisson constraint
in time.  The result is the time derivative of
\[
  \frac12\int_\Om
  \bigl(A\abs{\pa^\alpha\psi}^2
  +\eps^2\abs{\nabla\pa^\alpha\psi}^2\bigr)\,\dd x,
  \qquad A=e^{\phi^a+\psi},
\]
up to tame commutators.  The residual contribution is closed with
$\cR_\phi^a\in H^{s+1}$, while the nonlinear screening commutator is closed
with only $\psi\in H^s$ and
$\eps\nabla\psi\in H^s$.  Thus no step divides by $\eps$.

This compensation, rather than smooth rarefaction stability by itself, is
the structural contribution of the paper.  It simultaneously handles a
nonconstant coefficient $e^{\phi^a+\psi}$, distinct neutral end states, the
transport of the top density derivative, fully two-dimensional perturbations,
and an arbitrary finite Debye hierarchy.  General smooth quasineutral
existence results supply important antecedents, but the estimate proved in
\Cref{lem:electric-compensation} and \Cref{prop:energy} is the new mechanism that makes
the conjunction \textup{(N1)}--\textup{(N6)} possible.

\subsection{Organization and analytical setting}

The remainder follows a theorem--proof format.  This section fixes the model,
state range, residuals, and relative energy.  The principal statements are
collected in \Cref{sec:main-results}.  Their profile, elliptic, energy, and
existence arguments are given in \Cref{sec:proofs};
\Cref{sec:conclusion} separates the result from the centered-fan problem, and
\Cref{sec:appendix} records the tame and cutoff estimates used in the proofs.

\subsection{Pressure and reference states}\label{sec:setting}

Fix integers $s\ge5$ and $M\ge0$.  We impose the following assumptions.

\begin{assumption}[Warm ions, compact state range, and genuine nonlinearity]\label{ass:pressure}
There are $0<n_*<n^*$ and a compact interval
$I^\sharp=[n_*/4,2n^*]\Subset(0,\infty)$ such that
$p_i\in C^{s+3M+6}(I^\sharp)$ and
\begin{equation}
  p_i'(z)\ge p_*>0\qquad(z\in I^\sharp). \label{eq:warm-ion}
\end{equation}
The two neutral end states belong to $(n_*,n^*)$, and the rarefaction curve
joining them stays in a compact subinterval of $[n_*,n^*]$.
The eigenvalues of the one-dimensional effective Euler system in primitive
variables are
\[
  \lambda_\pm(N,V)=V\pm c_q(N),
  \qquad c_q(N)=\sqrt{p_i'(N)+1}.
\]
The end states are joined by a single genuinely nonlinear family, and the
rarefaction strength
\[
  \delta=\abs{n_+-n_-}+\abs{u_+-u_-}
\]
is at most a fixed $\delta_0$.  Smallness of $\delta_0$ is used only to keep
the profile in $[n_*,n^*]$ and to make constants uniform over the chosen
family.
\end{assumption}

\begin{definition}[Profile seminorm]\label{def:profile-seminorm}
For fixed $0\le t_0<T$ and a planar coefficient $f=f(t,x_1)$, set
\[
  \abs f_{\cX_{k,T}}
  =\sum_{a+b\le k}\sup_{t_0\le t\le T}
  \norm{\pa_t^a\pa_1^bf(t)}_{L^\infty(\R)}
  +\sum_{a+b\le k}\norm{\pa_t^a\pa_1^{b+1}f}_{L^2([t_0,T]\times\R)}.
\]
Only derivatives are placed in $L^2$, since the profile has different end
states.
\end{definition}

\subsection{Residuals, energy, and solution class}

For a smooth triple $(n^a,u^a,\phi^a)$ define
\begin{align}
  \cR_n^a&=\pa_tn^a+\Div(n^au^a), \label{eq:Rn}\\
  \cR_u^a&=\pa_tu^a+u^a\cdot\nabla u^a+\nabla h_i(n^a)+\nabla\phi^a, \label{eq:Ru}\\
  \cR_\phi^a&=\eps^2\Delta\phi^a-e^{\phi^a}+n^a. \label{eq:Rphi}
\end{align}
We use the residual norm
\begin{equation}
  \mathfrak R_s^a(t)=
  \norm{\cR_n^a}_{H^s}+\norm{\cR_u^a}_{H^s}
  +\norm{\cR_\phi^a}_{H^{s+1}}+\norm{\pa_t\cR_\phi^a}_{H^s}.
  \label{eq:residual-norm}
\end{equation}
For $W=(r,v,\psi)$, with $n=n^a+r$, set
\begin{equation}
\begin{split}
  \cE_{s,\eps}[W](t)
  ={}&\frac12\sum_{\abs\alpha\le s}\int_\Om
  \bigl(n\abs{\pa^\alpha v}^2
  +h_i'(n)\abs{\pa^\alpha r}^2\bigr)\,\dd x\\
  &+\int_\Om\left\{\frac{\eps^2}{2}\abs{\nabla\psi}^2
  +\cH_e^*(\psi;\phi^a)\right\}\,\dd x\\
  &+\frac12\sum_{1\le\abs\alpha\le s}\int_\Om
  \left(A\abs{\pa^\alpha\psi}^2
  +\eps^2\abs{\nabla\pa^\alpha\psi}^2\right)\,\dd x,
  \qquad A=e^{\phi^a+\psi}, \label{eq:energy-definition}
\end{split}
\end{equation}
where
\begin{equation}
  \cH_e^*(\psi;\phi^a)
  =e^{\phi^a}\bigl(e^\psi(\psi-1)+1\bigr). \label{eq:dual-electron-setting}
\end{equation}
This separation of the zeroth electric entropy from the differentiated
quadratic forms is important: the coefficient of a top derivative is
$e^{\phi^a+\psi}$, whereas the exact undifferentiated entropy is
\eqref{eq:dual-electron-setting}.  Spatial derivatives are used in
\eqref{eq:energy-definition}.  Time derivatives follow from the equations
once the spatial estimate is known.

The natural perturbation space is the $\eps$-dependent graph space
\begin{equation}
  \cG^s_\eps
  =\{(r,v,\psi):r,v,\psi\in H^s(\Om),\ 
  \eps\nabla\psi\in H^s(\Om)\}, \label{eq:graph-space}
\end{equation}
equipped with
$\norm r_{H^s}+\norm v_{H^s}+\norm\psi_{H^s}
+\eps\norm{\nabla\psi}_{H^s}$.  Uniformity below always refers to this
norm.

By a classical perturbative solution on $[t_0,T]$ we mean a triple in
$C([t_0,T];\cG^s_\eps)$ whose fluid variables are in
$C^1([t_0,T];H^{s-1})$, whose potential is in
$C^1([t_0,T];H^{s-1})$ with $\eps\nabla\pa_t\psi\in H^{s-1}$, and which
satisfies the equations pointwise.  Since $s\ge5$ and the space dimension is
two, these regularities are more than sufficient for the products and
characteristic flow used below.

\section{Main Results}\label{sec:main-results}

The first theorem establishes uniform stability for perturbations whose size is
strictly independent of the Debye length. Preparedness is required only to extract
a quantitative asymptotic expansion.

\begin{theorem}[Uniform stability of a smooth expanding simple wave]\label{thm:main}
Let Assumption~\ref{ass:pressure} hold, let $t_0\ge0$ and $T>t_0$ be fixed, and let
$(\Nr,\Vr)$ be the smooth planar simple wave constructed in
\Cref{prop:smooth-simple-wave}.  For every $M\ge0$ there are
$\eps_0,\eta_0>0$ and a composite profile
$(n^a,u^a,\phi^a)$ of the form \eqref{eq:composite-intro} such that the
following holds for $0<\eps\le\eps_0$.

Suppose the initial density and velocity at $t=t_0$ satisfy the far-field
conditions, $n_0(x)\in[n_*/2,3n^*/2]$, and
\[
  \cE_{s,\eps}[r_0,v_0,\psi_0](t_0)\le\eta_0^2,
\]
where $\phi_0=\phi^a(t_0)+\psi_0$ is the unique solution of
\eqref{eq:EP-poisson}.  Then \eqref{eq:EP-mass}--\eqref{eq:EP-poisson}
has a unique classical solution on $[t_0,T]$ with
\[
  (r,v,\psi)\in C([t_0,T];H^s(\Om)),
  \qquad \eps\nabla\psi\in C([t_0,T];H^s(\Om)).
\]
Uniqueness holds among classical solutions with this regularity, positive
density in $[n_*/4,2n^*]$, the prescribed far fields, and values in the
Poisson chart \eqref{eq:poisson-map}; by continuity any classical solution
issued from the stated data lies in that class on a maximal common interval,
where \Cref{lem:difference} applies.
Moreover,
\begin{equation}
  \sup_{t_0\le t\le T}\cE_{s,\eps}(t)
  \le C_{t_0,T}\bigl(\cE_{s,\eps}(t_0)+\eps^{4M+4}\bigr), \label{eq:main-estimate}
\end{equation}
where $C_{t_0,T}$ and the smallness thresholds are independent of $\eps$.
The density stays in $[n_*/4,2n^*]$.
\end{theorem}

The main estimate successfully governs initial perturbations independent of $\eps$. If the
data are prepared to the order of the profile, the theorem yields a direct asymptotic
expansion.

\begin{corollary}[Arbitrary-order quasineutral expansion]\label{cor:expansion}
Under the hypotheses of \Cref{thm:main}, assume
\[
  \cE_{s,\eps}(t_0)\le C_0\eps^{4M+4}.
\]
Then
\begin{align}
  &\sup_{t_0\le t\le T}
  \left(
  \norm{n-n^a}_{H^s}+\norm{u-u^a}_{H^s}
  +\norm{\phi-\phi^a}_{H^s}
  +\eps\norm{\nabla(\phi-\phi^a)}_{H^s}
  \right)
  \le C\eps^{2M+2}. \label{eq:expansion-estimate}
\end{align}
In particular, for every $M\ge0$,
\begin{equation}
  \sup_{t_0\le t\le T}
  \left(
    \norm{n-\Nr}_{H^s}+\norm{u-(\Vr,0)}_{H^s}
    +\norm{\phi-\log\Nr}_{H^s}
  \right)\le C\eps^2. \label{eq:leading-rate}
\end{equation}
Here each norm is taken after subtracting the corresponding composite or
simple-wave profile, rendering it finite over the distinct end states.
\end{corollary}

\begin{remark}[Smooth-wave scope]\label{rem:smooth-wave-scope}
Because the Burgers datum $w_0$ is smooth and monotone, the reference wave is
globally smooth, permitting $t_0=0$. This establishes rigorous stability for the regular 
characteristic geometry shown in \Cref{fig:characteristics}(a). This framework isolates 
the stability of the singular Debye limit from the distinct geometric challenges 
posed by the Riemann vertex and non-smooth fan edges characteristic of the 
centered fan shown in \Cref{fig:characteristics}(b).
\end{remark}

\begin{figure}[htbp]
\centering
\includegraphics[width=\textwidth]{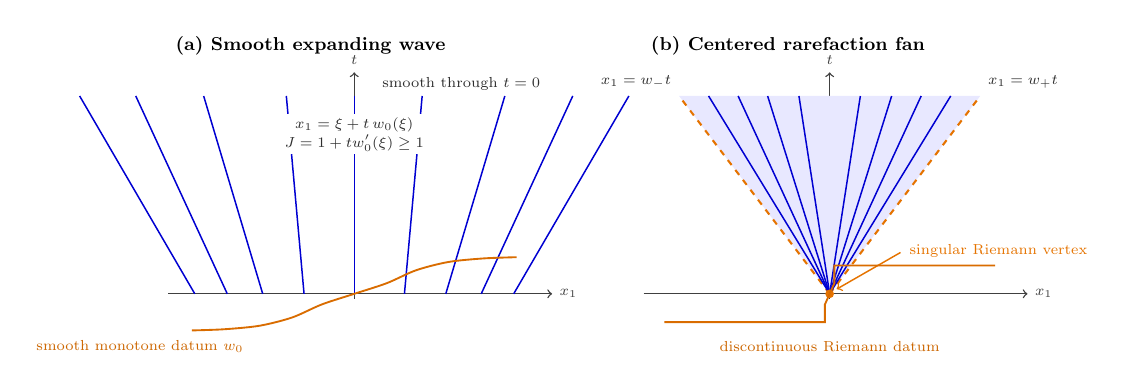}
\caption{Characteristic geometry and scope. The smooth monotone Burgers
datum in panel (a) generates a globally regular characteristic map with
$J=1+t w_0'\ge1$, forming the foundational reference wave for \Cref{thm:main}. 
In contrast, panel (b) illustrates the centered fan, which possesses the same limiting
states but introduces geometric singularities at the Riemann vertex and non-smooth 
edges. \Cref{prop:centered-comparison} quantifies the precise large-time convergence 
between these two regimes.}
\label{fig:characteristics}
\end{figure}

\section{Proofs}\label{sec:proofs}

\subsection{Smooth expanding simple wave and Debye composite}\label{sec:profile}

\subsubsection{A smooth exact simple wave of the effective Euler system}

Let $k\in\{-,+\}$ be the selected characteristic family.  Along its
rarefaction curve the other Riemann invariant is constant and
$\lambda_k(N,V)$ is a smooth coordinate.  Choose a monotone
$w_0\in C^\infty(\R)$ satisfying
\[
  w_0(\pm\infty)=\lambda_k(n_\pm,u_\pm),\qquad
  w_0'\ge0,\qquad
  w_0'\in H^{s+3M+5}(\R),
\]
and, without loss of generality, choose exponential tails
\begin{equation}
 \abs{\pa_1^j(w_0-w_\pm)}\le C_j\Delta w\,e^{-c\abs{x_1}}
 \quad\text{on }\{\pm x_1\ge0\},\qquad
 0\le j\le s+3M+6,\qquad
 \Delta w=w_+-w_->0, \label{eq:w0-tails}
\end{equation}
where $w_\pm=\lambda_k(n_\pm,u_\pm)$.  Such a profile is obtained, for
example, from an affine rescaling of $\tanh x_1$.  We solve
\begin{equation}
  \pa_tw+w\pa_1w=0,\qquad w(0,x_1)=w_0(x_1). \label{eq:burgers}
\end{equation}
The characteristic formula
$w(t,x_1)=w_0(\xi)$, $x_1=\xi+t w_0(\xi)$ gives a global smooth solution.
This Burgers-generated transition, rather than the piecewise self-similar
Riemann fan, is the reference class throughout the paper; we call it a smooth
expanding planar simple wave.  Its expansion is encoded by
$\pa_\xi(\xi+t w_0(\xi))=1+t w_0'(\xi)\ge1$.
Define $(\Nr,\Vr)$ on the rarefaction curve by
\begin{equation}
  \lambda_k(\Nr,\Vr)=w,
  \qquad R_{-k}(\Nr,\Vr)=R_{-k}(n_-,u_-). \label{eq:rare-invariants}
\end{equation}

\begin{proposition}[Smooth expanding planar simple wave]\label{prop:smooth-simple-wave}
The pair $(\Nr,\Vr)$ defined by \eqref{eq:burgers}--\eqref{eq:rare-invariants}
is an exact smooth planar solution of \eqref{eq:QE-mass}--\eqref{eq:QE-momentum}.
It has the prescribed end states, stays in $[n_*,n^*]$, and, for each fixed
$0\le t_0<T$,
\begin{equation}
  \abs{(\Nr,\Vr)}_{\cX_{s+3M+5,T}}\le C_{t_0,T}\delta. \label{eq:rare-bounds}
\end{equation}
The constant is uniform for end states in a fixed compact subset of the
genuinely nonlinear rarefaction regime.
\end{proposition}

\begin{proof}
Let $\Xi_t(\xi)=\xi+tw_0(\xi)$ and $J(t,\xi)=\pa_\xi\Xi_t
=1+tw_0'(\xi)$.  Monotonicity gives $J\ge1$.  Since
$\Xi_t(\xi)\to\pm\infty$ as $\xi\to\pm\infty$, $\Xi_t$ is a global
$C^\infty$ diffeomorphism and
\[
 w(t,x_1)=w_0(\Xi_t^{-1}(x_1)).
\]
Differentiating $x_1=\Xi_t(\xi)$ gives the two basic identities
\begin{equation}
 \pa_1w=\frac{w_0'(\xi)}{J(t,\xi)},
 \qquad
 \pa_tw=-\frac{w_0(\xi)w_0'(\xi)}{J(t,\xi)}=-w\pa_1w.
 \label{eq:burgers-first-derivatives}
\end{equation}
In particular $\pa_1w\ge0$.  An induction using
$\pa_1=J^{-1}\pa_\xi$ shows that, for $m\ge1$,
\begin{equation}
 \pa_1^mw(t,x_1)
 =\sum_{\nu=1}^m
 \frac{t^{\nu-1}P_{m,\nu}
 (w_0'(\xi),\ldots,w_0^{(m)}(\xi))}
 {J(t,\xi)^{m+\nu-1}},
 \label{eq:burgers-spatial-derivatives}
\end{equation}
where every monomial in $P_{m,\nu}$ contains at least one derivative of
$w_0$.  Mixed time-space derivatives are obtained by repeatedly replacing
$\pa_tw$ with $-w\pa_1w$; hence they have the same structure, with bounded
undifferentiated factors $w$.

Because $J\ge1$ and $\dd x_1=J\,\dd\xi$, the exponential localization in
\eqref{eq:w0-tails} and \eqref{eq:burgers-spatial-derivatives} imply, for
$a+b\le s+3M+5$,
\begin{equation}
\begin{aligned}
 \sup_{t_0\le t\le T}
 \norm{\pa_t^a\pa_1^bw(t)}_{L^\infty}
 &\le C_{a,b,T}\Delta w,\\
 \norm{\pa_t^a\pa_1^{b+1}w}_{L^2([t_0,T]\times\R)}
 &\le C_{a,b,t_0,T}\Delta w.
\end{aligned}
\label{eq:w-profile-bounds}
\end{equation}
No singular factor occurs when $t_0=0$; the constants are controlled by the
fixed smooth datum $w_0$.

Let $\Gamma_k(w)=(N(w),V(w))$ denote the inverse of
$(N,V)\mapsto(\lambda_k,R_{-k})$ on the compact rarefaction curve.  Its
derivative is a nonzero multiple of the $k$th right eigenvector.  Therefore
$U^r=\Gamma_k(w)$ satisfies
\[
 \pa_tU^r+A(U^r)\pa_1U^r
 =\Gamma_k'(w)(\pa_tw+\lambda_k(U^r)\pa_1w)=0,
\]
which is precisely the one-dimensional form of
\eqref{eq:QE-mass}--\eqref{eq:QE-momentum}.  The end states and state-range
claims follow from the definition of $\Gamma_k$.  Finally, the ordinary
chain rule, \eqref{eq:w-profile-bounds}, and boundedness of the
derivatives of $\Gamma_k$ give \eqref{eq:rare-bounds}.
\end{proof}

Let
\begin{equation}
 w^c(\xi)=
 \begin{cases}
  w_-,&\xi\le w_-,\\
  \xi,&w_-<\xi<w_+,\\
  w_+,&\xi\ge w_+,
 \end{cases} \label{eq:centered-burgers-fan}
\end{equation}
and let $(N^c,V^c)(\xi)$ be its image on the same rarefaction curve.

\begin{proposition}[Comparison with the centered fan]
\label{prop:centered-comparison}
For the exponentially localized choice \eqref{eq:w0-tails},
\begin{equation}
 \sup_{x_1\in\R}
 \abs{(\Nr,\Vr)(t,x_1)-(N^c,V^c)(x_1/t)}
 \le C\frac{\Delta w}{1+t\Delta w}
 \log(2+t\Delta w),\qquad t>0. \label{eq:centered-comparison}
\end{equation}
In particular, the smooth reference converges uniformly to the centered
Euler rarefaction as $t\to\infty$.
\end{proposition}

\begin{proof}
Fix $t>0$ and write $x_1=\xi+t w_0(\xi)$.  We first assume that
$x_1/t\in(w_-,w_+)$.  Since $w^c(x_1/t)=x_1/t$, the characteristic identity
gives
\begin{equation}
  w^c(x_1/t)-w(t,x_1)=\xi/t. \label{eq:fan-interior-difference}
\end{equation}
The inequalities $w_-<w_0(\xi)+\xi/t<w_+$ imply
\[
 -t\bigl(w_0(\xi)-w_-\bigr)<\xi
 <t\bigl(w_+-w_0(\xi)\bigr).
\]
Since both differences in parentheses lie in $[0,\Delta w]$, the same
inequality gives the small-time bound $\abs\xi\le t\Delta w$.
For $\xi\ge0$, the right tail in \eqref{eq:w0-tails} gives
$\xi\le Ct\Delta w e^{-c\xi}$; for $\xi\le0$, the left tail gives the
corresponding inequality for $-\xi$.  The elementary implication
$y\le ae^{-cy}\Rightarrow y\le C_c\log(2+a)$ yields
\begin{equation}
 \abs\xi\le C\log(2+t\Delta w) 
 \quad\text{whenever }x_1/t\in(w_-,w_+). \label{eq:xi-log-bound}
\end{equation}
Combining these two bounds with \eqref{eq:fan-interior-difference} gives
\[
 \abs{w^c(x_1/t)-w(t,x_1)}
 \le C\min\left\{\Delta w,
       \frac{\log(2+t\Delta w)}{t}\right\},
\]
which proves the desired estimate in the fan interior in both regimes
$t\Delta w\le1$ and $t\Delta w\ge1$.

If $x_1/t\ge w_+$, then $w^c(x_1/t)=w_+$ and
$x_1/t=w_0(\xi)+\xi/t\ge w_+$.  Since $w_0\le w_+$, this inequality forces
$\xi\ge0$, and the right tail gives
\[
 0\le w_+-w(t,x_1)=w_+-w_0(\xi)\le C\Delta w e^{-c\xi}.
\]
Its maximum subject to
$\xi/t\ge w_+-w_0(\xi)$ is bounded by
$C\min\{\Delta w,t^{-1}\log(2+t\Delta w)\}$.  The region
$x_1/t\le w_-$ is treated identically.  We have therefore shown
\[
 \norm{w(t,\cdot)-w^c(\cdot/t)}_\infty
 \le C\min\left\{\Delta w,\frac{\log(2+t\Delta w)}{t}\right\}
 \le C\frac{\Delta w}{1+t\Delta w}\log(2+t\Delta w).
\]
The map $\Gamma_k$ used in the previous proof is Lipschitz on the compact
rarefaction curve.  Applying it to the last scalar estimate proves
\eqref{eq:centered-comparison}.
\end{proof}

\begin{remark}
The theorem establishes stability for the smooth exact simple wave, not for
the nonsmooth edges of \eqref{eq:centered-burgers-fan}.
\Cref{prop:centered-comparison} quantifies its precise large-time relation to the
centered fan.  Uniform control of the fan edges or the Riemann vertex would
require the geometric theory discussed in the introduction.
\end{remark}

\subsubsection{Corrector hierarchy}

Write $U_0=(\Nr,(\Vr,0))$ and $\Phi_0=\log\Nr$.  We seek
\begin{equation}
  U^a=\sum_{j=0}^M\eps^{2j}U_j,
  \qquad
  \phi^a=\sum_{j=0}^M\eps^{2j}\Phi_j. \label{eq:corrector-ansatz}
\end{equation}
For a formal series $G(z)$, denote its $z^j$ coefficient by $[z^j]G$.
Writing
\[
 \widehat N=\sum_{j\ge0}z^jN_j,\qquad
 \widehat V=\sum_{j\ge0}z^jV_j,\qquad
 \widehat\Phi=\sum_{j\ge0}z^j\Phi_j,
\]
the hierarchy is defined without ambiguity by
\begin{align}
 [z^j]\{\pa_t\widehat N+\pa_1(\widehat N\widehat V)\}&=0,
 \label{eq:coefficient-mass}\\
 [z^j]\{\pa_t\widehat V+\widehat V\pa_1\widehat V
 +h_i'(\widehat N)\pa_1\widehat N+\pa_1\widehat\Phi\}&=0,
 \label{eq:coefficient-momentum}\\
 [z^j]\{z\pa_1^2\widehat\Phi-e^{\widehat\Phi}+\widehat N\}&=0,
 \label{eq:coefficient-poisson}
\end{align}
for $0\le j\le M$.  At each fixed order only finitely many Taylor
coefficients occur, so these identities are algebraic definitions even
though the displayed series are formal.

At order $\eps^2$, set $U_1=(N_1,(V_1,0))$.  The Poisson equation gives
\begin{equation}
  \Phi_1=\frac{N_1+\pa_1^2\log\Nr}{\Nr}. \label{eq:first-phi}
\end{equation}
The fluid equations give the forced linearized effective Euler system
\begin{align}
  \pa_tN_1+\pa_1(\Vr N_1+\Nr V_1)&=0, \label{eq:first-corrector-mass}\\
  \pa_tV_1+\Vr\pa_1V_1+V_1\pa_1\Vr
  +\pa_1\bigl(q'(\Nr)N_1\bigr)
  &=-\pa_1\left(\frac{\pa_1^2\log\Nr}{\Nr}\right). \label{eq:first-corrector-momentum}
\end{align}
We take $(N_1,V_1)(t_0)=0$.  Higher orders have the same principal part.

To make the induction explicit, put $Y_j=(N_j,V_j)^{\mathsf T}$ and define
\begin{equation}
 \mathcal L_rY_j=
 \begin{pmatrix}
  \pa_tN_j+\Vr\pa_1N_j+\Nr\pa_1V_j
    +(\pa_1\Vr)N_j+(\pa_1\Nr)V_j\\
  \pa_tV_j+\Vr\pa_1V_j+q'(\Nr)\pa_1N_j
    +(\pa_1\Vr)V_j+q''(\Nr)(\pa_1\Nr)N_j
 \end{pmatrix}. \label{eq:linearized-rarefaction}
\end{equation}
At order $\eps^{2j}$, the hierarchy has the triangular form
\begin{equation}
  \Nr\Phi_j-N_j=S_j,\qquad
  \mathcal L_rY_j=F_j,\qquad Y_j(t_0)=0, \label{eq:hierarchy-abstract}
\end{equation}
where $S_j$ and $F_j$ depend only on $U_0,\ldots,U_{j-1}$ and
$\Phi_0,\ldots,\Phi_{j-1}$.  In particular, substitution of
$\Phi_j=(N_j+S_j)/\Nr$ moves every occurrence of the unknown $N_j$ into the
effective pressure coefficient $q'(\Nr)$.  This is the algebraic reason that
the fluid equation at every order has the same strictly hyperbolic principal
part.  The associated derivative budget is summarized in
\Cref{fig:hierarchy}.

\begin{figure}[htbp]
\centering
\includegraphics[width=\linewidth]{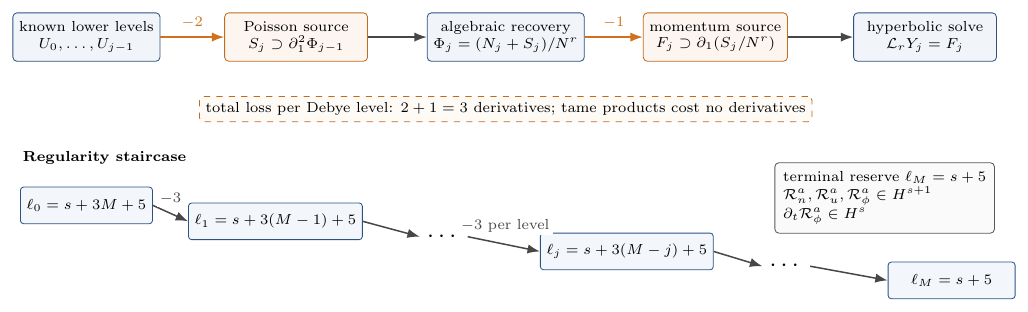}
\caption{Derivative bookkeeping in the finite Debye hierarchy.  At level
$j$, the Poisson source $S_j$ costs two derivatives and its re-entry into the
momentum forcing costs one more; the remaining tame products and the
symmetrizable hyperbolic solve cause no additional loss.  Consequently
$\ell_j=s+3(M-j)+5$, leaving the terminal reserve needed for
$\mathcal R_\phi^a\in H^{s+1}$ and
$\partial_t\mathcal R_\phi^a\in H^s$.}
\label{fig:hierarchy}
\end{figure}

\begin{lemma}[Uniform remainder of a finite Debye hierarchy]
\label{lem:profile-remainder}
For $z\in[0,z_0]$, define
\[
 N_M(z)=\sum_{j=0}^Mz^jN_j,\qquad
 V_M(z)=\sum_{j=0}^Mz^jV_j,\qquad
 \Phi_M(z)=\sum_{j=0}^Mz^j\Phi_j,
\]
and put $\ell_j=s+3(M-j)+5$ for $0\le j\le M$.
Assume that the coefficient identities
\eqref{eq:coefficient-mass}--\eqref{eq:coefficient-poisson} hold through
order $M$, that the base profile has the bounds of
\Cref{prop:smooth-simple-wave}, and that
\begin{equation}
\begin{split}
 \mathcal B_M={}&
 \sum_{j=1}^M\sup_{t_0\le t\le T}\bigl(
 \norm{(N_j,V_j,\Phi_j)(t)}_{H^{\ell_j}}
 +\norm{\pa_t(N_j,V_j,\Phi_j)(t)}_{H^{\ell_j-1}}\bigr)<\infty.
 \label{eq:remainder-coefficient-bound}
\end{split}
\end{equation}
Suppose also that the positive-order coefficients and their derivatives tend
to zero at spatial infinity.  Then, uniformly for $0\le z\le z_0$,
\begin{equation}
\begin{split}
 &\norm{\mathcal F_n(z)}_{H^{s+1}}
 +\norm{\mathcal F_u(z)}_{H^{s+1}}
 +\norm{\mathcal F_\phi(z)}_{H^{s+1}}
 +\norm{\pa_t\mathcal F_\phi(z)}_{H^s}
 \le C_{\mathcal B_M,t_0,T}z^{M+1},
 \label{eq:finite-hierarchy-remainder}
\end{split}
\end{equation}
where
\begin{align}
 \mathcal F_n(z)&=\pa_tN_M(z)+\pa_1(N_M(z)V_M(z)),
 \label{eq:remainder-mass-functional}\\
 \mathcal F_u(z)&=\pa_tV_M(z)+V_M(z)\pa_1V_M(z)
 +\pa_1h_i(N_M(z))+\pa_1\Phi_M(z),
 \label{eq:remainder-momentum-functional}\\
 \mathcal F_\phi(z)&=z\pa_1^2\Phi_M(z)
 -e^{\Phi_M(z)}+N_M(z).
 \label{eq:remainder-poisson-functional}
\end{align}
\end{lemma}

\begin{proof}
For $\mathcal F\in\{\mathcal F_n,\mathcal F_u,\mathcal F_\phi\}$, the
coefficient identities say
\begin{equation}
 \pa_z^j\mathcal F(0)=0,\qquad 0\le j\le M.
 \label{eq:remainder-vanishing-coefficients}
\end{equation}
Taylor's formula therefore has no finite part:
\begin{equation}
 \mathcal F(z)=\frac{z^{M+1}}{M!}
 \int_0^1(1-\theta)^M
 \pa_z^{M+1}\mathcal F(\theta z)\,\dd\theta.
 \label{eq:profile-Taylor-remainder}
\end{equation}
It remains to bound the integral in the four norms in
\eqref{eq:finite-hierarchy-remainder}.

The mass functional is polynomial.  After $M+1$ derivatives in $z$, every
surviving term has the form
\begin{equation}
 c_{jk}\,\pa_1(N_jV_k),\qquad
 0\le j,k\le M,\qquad j+k\ge M+1.
 \label{eq:mass-remainder-monomial}
\end{equation}
Thus at least one index is positive, so the term is spatially localized.
The $H^{s+1}$ product estimate costs one derivative and is bounded by
\eqref{eq:remainder-coefficient-bound}, because
$\ell_M=s+5$.
The quadratic velocity term in $\mathcal F_u$ is identical.

For the pressure and exponential compositions, the parameter derivative has
the explicit Fa\`a di Bruno form
\begin{equation}
 \pa_z^kG(X_M(z))
 =\sum_{\ell=1}^kG^{(\ell)}(X_M(z))
 \sum_{\substack{r_1+\cdots+r_\ell=k\\r_i\ge1}}
 c_{\mathbf r}\prod_{i=1}^{\ell}\pa_z^{r_i}X_M(z),
 \qquad
 \pa_z^rX_M(z)=
 \sum_{j=r}^M\frac{j!}{(j-r)!}z^{j-r}X_j,
 \label{eq:parameter-Faa-di-Bruno}
\end{equation}
with $G=h_i$ and $X_M=N_M$, or with $G=e^{(\cdot)}$ and
$X_M=\Phi_M$.  If $k=M+1$, each nonzero product in
\eqref{eq:parameter-Faa-di-Bruno} contains a positive-order coefficient.
The state range is compact, so all derivatives of $G$ that occur are
bounded.  The tame product estimate in \Cref{lem:appendix-tame} places the
factor with the largest spatial derivative in $L^2$ and all remaining
factors in $L^\infty$.  Consequently,
\begin{align}
 \sup_{0\le z\le z_0}
 \norm{\pa_z^{M+1}\pa_1h_i(N_M(z))}_{H^{s+1}}
 &\le C_{\mathcal B_M},\label{eq:pressure-remainder-bound}\\
 \sup_{0\le z\le z_0}
 \norm{\pa_z^{M+1}e^{\Phi_M(z)}}_{H^{s+1}}
 &\le C_{\mathcal B_M}.
 \label{eq:exponential-remainder-bound}
\end{align}
The first estimate uses at most $s+2$ spatial derivatives of $N_j$.
The explicit term $z\pa_1^2\Phi_M$ uses at most $s+3$ derivatives of
$\Phi_j$.  Both are below the lowest reserve
$\ell_M=s+5$.  Equations
\eqref{eq:mass-remainder-monomial}--\eqref{eq:exponential-remainder-bound}
prove the three spatial bounds in
\eqref{eq:finite-hierarchy-remainder}.

For the last bound, differentiate
\eqref{eq:remainder-poisson-functional} in time before applying Taylor's
formula:
\begin{equation}
 \pa_t\mathcal F_\phi(z)
 =z\pa_1^2\pa_t\Phi_M(z)
 -e^{\Phi_M(z)}\pa_t\Phi_M(z)+\pa_tN_M(z).
 \label{eq:time-remainder-functional}
\end{equation}
After $M+1$ parameter derivatives, Leibniz' rule and
\eqref{eq:parameter-Faa-di-Bruno} produce products with exactly one
time-differentiated coefficient and only spatially differentiated
coefficients otherwise.  Estimate
\eqref{eq:remainder-coefficient-bound} gives
$\pa_t\Phi_j\in H^{\ell_j-1}$ and
$\pa_tN_j\in H^{\ell_j-1}$.  The worst explicit term in
\eqref{eq:time-remainder-functional} requires
$s+2$ spatial derivatives of $\pa_t\Phi_j$, while
$\ell_M-1=s+4$.  The same high--low allocation therefore gives
\begin{equation}
 \sup_{0\le z\le z_0}
 \norm{\pa_z^{M+1}\pa_t\mathcal F_\phi(z)}_{H^s}
 \le C_{\mathcal B_M,t_0,T}.
 \label{eq:time-remainder-bound}
\end{equation}
Every remainder monomial contains a positive-order coefficient or a spatial
derivative of the base wave; hence it belongs to $L^2$ despite the distinct
end states.  Substitution of
\eqref{eq:pressure-remainder-bound}--\eqref{eq:time-remainder-bound} into
\eqref{eq:profile-Taylor-remainder} proves
\eqref{eq:finite-hierarchy-remainder}.
\end{proof}

\begin{proposition}[Composite profile]\label{prop:composite}
For each $M\ge0$ there are unique planar correctors
$(U_j,\Phi_j)_{1\le j\le M}$, with zero fluid data at $t=t_0$ and zero
far-field limits, such that the profile \eqref{eq:corrector-ansatz} obeys
\begin{equation}
  \sup_{0<\eps\le\eps_0}
  \bigl(\abs{n^a}_{\cX_{s+3,T}}+\abs{u^a}_{\cX_{s+3,T}}
  +\abs{\phi^a}_{\cX_{s+3,T}}\bigr)\le C_{M,t_0,T}, \label{eq:profile-bound}
\end{equation}
and, for $\eps_0$ small,
$n^a\in[3n_*/4,5n^*/4]$.  Its residual satisfies
\begin{equation}
\begin{split}
  &\sup_{t_0\le t\le T}
  \bigl(\norm{\cR_n^a}_{H^{s+1}}+\norm{\cR_u^a}_{H^{s+1}}
  +\norm{\cR_\phi^a}_{H^{s+1}}
  +\norm{\pa_t\cR_\phi^a}_{H^s}\bigr)
  \le C_{M,t_0,T}\eps^{2M+2}. \label{eq:residual-bound}
\end{split}
\end{equation}
All residuals and all corrector derivatives vanish at spatial infinity.
\end{proposition}

\begin{proof}
Put $z=\eps^2$.  We first verify the triangular structure and the derivative
count.  Since $e^{\Phi_0}=\Nr$, the coefficient of $z^j$ in the Poisson
equation is
\begin{equation}
  \Nr\Phi_j-N_j=S_j,
  \qquad
  S_j=\pa_1^2\Phi_{j-1}
  -\Nr\sum_{\ell=2}^j
  \sum_{\substack{j_1+\cdots+j_\ell=j\\j_1,\ldots,j_\ell\ge1}}
  \frac{\Phi_{j_1}\cdots\Phi_{j_\ell}}{\ell!}, \label{eq:triangular-poisson}
\end{equation}
with the convention $\Phi_{-1}=0$.  Every index in the double sum is at
most $j-1$, so $S_j$ is already known at the $j$th step and
$\Phi_j=(N_j+S_j)/\Nr$.

For clarity, the mass component of the fluid forcing is
\begin{equation}
 (F_j)_1=-\sum_{\ell=1}^{j-1}
 \pa_1(N_\ell V_{j-\ell}). \label{eq:Fj-mass}
\end{equation}
After substituting $\Phi_j=(N_j+S_j)/\Nr$, the momentum component can be
written
\begin{equation}
 (F_j)_2=-\sum_{\ell=1}^{j-1}V_\ell\pa_1V_{j-\ell}
 -\pa_1\left(\frac{S_j}{\Nr}\right)-\mathcal Q_j,
 \label{eq:Fj-momentum}
\end{equation}
where $\mathcal Q_j$ is the coefficient of $z^j$ in
\[
 h_i'(\widehat N)\pa_1\widehat N
 -h_i'(\Nr)\pa_1\Nr
 -z^j\left(h_i'(\Nr)\pa_1N_j
 +h_i''(\Nr)N_j\pa_1\Nr\right)
\]
after the lower coefficients have been fixed.  Equivalently,
$\mathcal Q_j$ is a finite sum of products of lower $N_\ell$ and their first
derivatives.  The two terms linear in $N_j$, together with
$\pa_1(N_j/\Nr)$, are exactly
\[
 q'(\Nr)\pa_1N_j+q''(\Nr)(\pa_1\Nr)N_j.
\]
This proves \eqref{eq:hierarchy-abstract} and also shows directly that no
unknown of order $j$ occurs in $F_j$.

Set
\[
 \ell_j=s+3(M-j)+5.
\]
The term $\pa_1^2\Phi_{j-1}$ loses two derivatives in $S_j$, and
$-\pa_1(S_j/\Nr)$ loses one more when the Poisson correction re-enters the
momentum equation.  Thus the hierarchy loses three spatial derivatives per
Debye level.  All other operations in
\eqref{eq:triangular-poisson}--\eqref{eq:Fj-momentum} are lower-order tame
products in one space dimension.  Hence the
inductive bounds
\begin{equation}
 \norm{S_j}_{H^{\ell_j+1}}
 +\norm{F_j}_{H^{\ell_j}}
 \le C_{j,t_0,T}
 \left(1+\sum_{m<j}
 \bigl(\norm{Y_m}_{H^{\ell_m}}
   +\norm{\Phi_m}_{H^{\ell_m}}\bigr)^{P_j}\right)
 \label{eq:corrector-source-bound}
\end{equation}
hold for a finite integer $P_j$.  The coefficients in this estimate are
bounded by \eqref{eq:rare-bounds}.

We next solve the linear hyperbolic problem.  With the symmetrizer
\begin{equation}
  S_r=\begin{pmatrix}q'(\Nr)/\Nr&0\\0&1\end{pmatrix}, \label{eq:profile-symmetrizer}
\end{equation}
$S_rA_r$ is symmetric for
\[
 A_r=\begin{pmatrix}\Vr&\Nr\\q'(\Nr)&\Vr\end{pmatrix}.
\]
Both $S_r$ and $S_r^{-1}$ are uniformly bounded because
$q'(N)=(p_i'(N)+1)/N>0$ on the state interval.  For
$0\le m\le\ell_j$, apply $\pa_1^m$ to
$\mathcal L_rY_j=F_j$, multiply by $S_r\pa_1^mY_j$, and integrate.  The
principal term satisfies
\begin{align}
 \int \pa_1^mY_j^{\mathsf T}S_rA_r\pa_1^{m+1}Y_j
 =-\frac12\int \pa_1^mY_j^{\mathsf T}
 \pa_1(S_rA_r)\pa_1^mY_j. \label{eq:corrector-principal-ibp}
\end{align}
The remaining differentiated terms are commutators.  In one dimension,
\begin{equation}
 \norm{[\pa_1^m,a]\pa_1Y_j}_2
 \le C_m\norm{\pa_1a}_{W^{m-1,\infty}}
 \norm{Y_j}_{H^m},
 \qquad m\ge1, \label{eq:corrector-commutator}
\end{equation}
where the right-hand side follows directly by expanding the commutator and
placing the fixed coefficient in $L^\infty$.  Thus no derivative of
$Y_j$ above order $m$ occurs.  Summing in $m$ and using Young's inequality gives
\begin{equation}
  \frac{\dd}{\dd t}\norm{Y_j}_{H^\ell}^2
  \le C_{t_0,T}\norm{Y_j}_{H^\ell}^2
  +C\norm{F_j}_{H^\ell}^2,
  \qquad 0\le\ell\le s+3(M-j)+5. \label{eq:corrector-energy}
\end{equation}

For completeness, existence in \eqref{eq:corrector-energy} can be obtained
without invoking a black-box smooth theory.  Let $J_\nu$ be an even Fourier
cutoff in $x_1$ and solve
\[
 \pa_tY_{j,\nu}+J_\nu(A_r\pa_1J_\nu Y_{j,\nu}
 +B_rJ_\nu Y_{j,\nu})=J_\nu F_j,
 \qquad Y_{j,\nu}(t_0)=0.
\]
Here $B_r$ is the lower-order coefficient matrix in $\mathcal L_r$.
This is an ordinary differential equation on the band-limited subspace of
$H^{\ell_j}$.  Estimate \eqref{eq:corrector-energy} is unchanged up to a
constant independent of $\nu$, because $J_\nu$ is self-adjoint and commutes
with $\pa_1$; moving it through a variable coefficient produces only the
uniform commutator \eqref{eq:cutoff-commutator}.  The differentiated
transport right-hand side is uniformly
bounded only in $H^{\ell_j-1}$.  Accordingly, the difference of two
regularizations is estimated in $H^{\ell_j-2}$, where \Cref{lem:appendix-cutoff} makes the cutoff forcing tend to zero.  Thus
$Y_{j,\nu}$ is Cauchy in $C([t_0,T];H^{\ell_j-2})$, bounded weak-star in
$L^\infty H^{\ell_j}$, and, by interpolation, converges strongly in
$C H^{\ell'}$ for every $\ell'<\ell_j$.  Passing to the linear equation and
then applying its symmetrized energy identity to the limit gives
$Y_j\in C([t_0,T];H^{\ell_j})$.

Gronwall, \eqref{eq:corrector-source-bound}, and the zero initial data close
the induction through $j=M$.  The equation gives
$\pa_tY_j\in C H^{\ell_j-1}$.  From
$\Phi_j=(N_j+S_j)/\Nr$, the already constructed lower levels, and the two
derivatives displayed in \eqref{eq:triangular-poisson}, one obtains
\begin{equation}
 \Phi_j\in C H^{\ell_j},\qquad
 \pa_t\Phi_j\in C H^{\ell_j-1}.
 \label{eq:corrector-time-regularity}
\end{equation}

All terms in \eqref{eq:Fj-mass}--\eqref{eq:Fj-momentum} contain a derivative
of the simple wave or a lower corrector.  They therefore belong to $L^2(\R)$
and vanish at spatial infinity.  Since $Y_j,\Phi_j\in H^{\ell_j}(\R)$,
their continuous representatives and all
derivatives used in the construction tend to zero as $x_1\to\pm\infty$.
This proves the far-field assertion.

The bounds just proved are precisely
\eqref{eq:remainder-coefficient-bound}.  Apply
\Cref{lem:profile-remainder} with $z=\eps^2$.
Because $z^{M+1}=\eps^{2M+2}$, it gives all four estimates in
\eqref{eq:residual-bound}, including the $H^{s+1}$ Poisson residual and its
time derivative in $H^s$.  Finally,
$\norm{n^a-\Nr}_\infty\le C\eps^2$; choosing $\eps_0$ small gives the
stated density interval and completes the proof.
\end{proof}

\begin{remark}
The integer $M$ is not tied to the nonlinear argument.  Taking $M=0$ yields a
second-order residual and a first-order quasineutral theorem; increasing $M$
improves the expansion without changing the compensated energy.
\end{remark}

\subsection{The nonlinear screened Poisson resolver}\label{sec:poisson}

Let $(n^a,\phi^a)$ be the profile of \Cref{prop:composite}.  For a density
perturbation $r$, the potential perturbation is determined by
\begin{equation}
  -\eps^2\Delta\psi+e^{\phi^a}(e^\psi-1)=r+\cR_\phi^a,
  \qquad \psi(x_1,x_2)\to0\quad(x_1\to\pm\infty). \label{eq:nonlinear-poisson}
\end{equation}

For $m\ge0$ define
\begin{equation}
  \norm\psi_{\cG^m_\eps}
  =\norm\psi_{H^m}+\eps\norm{\nabla\psi}_{H^m}. \label{eq:potential-graph-norm}
\end{equation}

\begin{lemma}[Uniform screened inverse]\label{lem:linear-screened}
Let $a\in W^{m,\infty}(\Om)$ satisfy $0<a_*\le a\le a^*$.  For every
$f\in H^m(\Om)$ and $0<\eps\le1$, the equation
\begin{equation}
  -\eps^2\Delta z+az=f \label{eq:linear-screened}
\end{equation}
has a unique solution with
\begin{equation}
  \norm z_{\cG^m_\eps}\le
  C(a_*,\norm a_{W^{m,\infty}})\norm f_{H^m}. \label{eq:linear-screened-estimate}
\end{equation}
The constant does not depend on $\eps$.
\end{lemma}

\begin{proof}
On $H^1(\Om)$ consider
\[
 \mathfrak b_\eps(z,\chi)
 =\eps^2\ip{\nabla z}{\nabla\chi}+\ip{az}{\chi}.
\]
It is continuous and satisfies
$\mathfrak b_\eps(z,z)\ge
\eps^2\norm{\nabla z}_2^2+a_*\norm z_2^2$.
Lax--Milgram therefore gives a unique weak solution and, after pairing with
$z$ and applying Young's inequality,
\begin{equation}
 \norm z_2+\eps\norm{\nabla z}_2\le C(a_*)\norm f_2.
 \label{eq:linear-screened-order-zero}
\end{equation}

We justify the differentiated estimate first for smooth $a,f,z$.  For
$1\le\abs\alpha\le m$,
\begin{equation}
 -\eps^2\Delta\pa^\alpha z+a\pa^\alpha z
 =\pa^\alpha f-[\pa^\alpha,a]z. 
 \label{eq:linear-screened-commuted}
\end{equation}
Pairing with $\pa^\alpha z$ and using
\begin{equation}
 \norm{[\pa^\alpha,a]z}_2
 \le C\norm a_{W^{m,\infty}}\norm z_{H^{\abs\alpha-1}} \label{eq:linear-commutator}
\end{equation}
gives
\begin{align}
 a_*\norm{\pa^\alpha z}_2^2
 +\eps^2\norm{\nabla\pa^\alpha z}_2^2
 \le{}&\norm{\pa^\alpha f}_2\norm{\pa^\alpha z}_2
 +C\norm a_{W^{m,\infty}}
 \norm z_{H^{\abs\alpha-1}}\norm{\pa^\alpha z}_2.
 \label{eq:linear-screened-level}
\end{align}
Choose the Young parameter smaller than $a_*/4$, sum over the multi-indices
of a fixed order, and use \eqref{eq:linear-screened-order-zero}.  Induction
on the order yields
\begin{equation}
 \norm z_{H^k}^2+\eps^2\norm{\nabla z}_{H^k}^2
 \le C_k\norm f_{H^k}^2,
 \qquad 0\le k\le m, \label{eq:linear-screened-induction}
\end{equation}
where $C_k$ is independent of $\eps$.

For nonsmooth data, take $f_\nu\in C_c^\infty(\R\times\T)$ converging to
$f$ in $H^m$ and smooth coefficients $a_\nu$ with the same positive lower
bound and uniformly bounded $W^{m,\infty}$ norm.  The corresponding weak
solutions are Cauchy in $H^1$ by
\eqref{eq:linear-screened-order-zero} and uniformly bounded by
\eqref{eq:linear-screened-induction}.  Difference quotients in $x_1$ and
$x_2$, followed by weak lower semicontinuity, pass the $H^m$ estimate to the
limit on the cylinder.  The same order-zero pairing applied to the
difference of two solutions proves uniqueness.
\end{proof}

\begin{proposition}[Uniform nonlinear resolver]\label{prop:poisson-resolver}
Let $s\ge3$.  There are $\eta_P>0$ and $C>0$, independent of
$0<\eps\le\eps_0$, with the following property.  If
$r\in H^s(\Om)$, $\norm r_{H^s}\le\eta_P$, and $n^a+r$ stays positive,
then \eqref{eq:nonlinear-poisson} has a unique solution
$\psi\in H^s(\Om)$ with $\eps\nabla\psi\in H^s(\Om)$ and
\begin{equation}
  \norm\psi_{H^s}+\eps\norm{\nabla\psi}_{H^s}
  \le C\bigl(\norm r_{H^s}+\norm{\cR_\phi^a}_{H^s}\bigr). \label{eq:poisson-estimate}
\end{equation}
The map $r\mapsto\psi$ is locally Lipschitz in the same norms.  Moreover, if
$r\in H^{s}$ and the fluid equations give $\pa_tr\in H^{s-1}$, then the
time-differentiated equation gives
\begin{equation}
  \norm{\pa_t\psi}_{H^{s-1}}
  +\eps\norm{\nabla\pa_t\psi}_{H^{s-1}}
  \le C\bigl(\norm{\pa_tr}_{H^{s-1}}+\norm r_{H^s}
  +\norm{\pa_t\cR_\phi^a}_{H^{s-1}}\bigr). \label{eq:poisson-time}
\end{equation}
\end{proposition}

\begin{proof}
Let $L_\eps=-\eps^2\Delta+e^{\phi^a}$ and $f=r+\cR_\phi^a$.  The equation is
equivalent to the fixed-point problem
\begin{equation}
  \psi=L_\eps^{-1}\left[f-e^{\phi^a}
  (e^\psi-1-\psi)\right]=:\mathcal T_f(\psi). \label{eq:poisson-fixed-point}
\end{equation}
The profile bounds imply
$e^{\phi^a}\ge a_*>0$ and
$\norm{e^{\phi^a}}_{W^{s,\infty}}\le C$, uniformly in $\eps$.
For $s\ge3$, $H^s(\Om)$ is a Banach algebra and embeds into
$W^{1,\infty}(\Om)$.  Taylor's formula
\begin{equation}
 e^z-1-z=z^2\int_0^1(1-\theta)e^{\theta z}\,\dd\theta
 \label{eq:exp-quadratic-remainder}
\end{equation}
and the Moser composition estimate give, whenever
$\norm{\psi_i}_{H^s}\le1$,
\begin{align}
 \norm{e^\psi-1-\psi}_{H^s}
 &\le C\norm\psi_{H^s}^2, \label{eq:exp-quadratic-Moser}\\
 \norm{(e^{\psi_1}-1-\psi_1)
 -(e^{\psi_2}-1-\psi_2)}_{H^s}
 &\le C(\norm{\psi_1}_{H^s}+\norm{\psi_2}_{H^s})
 \norm{\psi_1-\psi_2}_{H^s}. 
 \label{eq:exp-Lipschitz-Moser}
\end{align}
Consequently, \Cref{lem:linear-screened} yields
\begin{align}
 \norm{\mathcal T_f(\psi)}_{\cG^s_\eps}
 &\le C\bigl(\norm f_{H^s}+\norm\psi_{H^s}^2\bigr),
 \label{eq:map-ball}\\
 \norm{\mathcal T_f(\psi_1)-\mathcal T_f(\psi_2)}_{\cG^s_\eps}
 &\le C(\norm{\psi_1}_{H^s}+\norm{\psi_2}_{H^s})
 \norm{\psi_1-\psi_2}_{\cG^s_\eps}. \label{eq:map-contraction}
\end{align}
Let $C_L$ be the larger constant in
\eqref{eq:map-ball}--\eqref{eq:map-contraction} and set
$R_f=2C_L\norm f_{H^s}$.  If $R_f\le1$ and $2C_LR_f\le1/2$, then
$\mathcal T_f$ maps the closed graph-norm ball of radius $R_f$ into itself
and has Lipschitz constant at most $1/2$.  Banach's theorem gives a unique
fixed point in this ball and proves \eqref{eq:poisson-estimate}.  Notice that
the smallness threshold depends on the profile bounds but not on $\eps$.

Uniqueness does not depend on the fixed-point representation.  If
$\psi_1,\psi_2$ are two small solutions and $z=\psi_1-\psi_2$, then
\begin{equation}
 -\eps^2\Delta z+e^{\phi^a}
 \left(\int_0^1e^{\theta\psi_1+(1-\theta)\psi_2}\,\dd\theta\right)z=0.
 \label{eq:poisson-uniqueness}
\end{equation}
The coefficient is positive and bounded away from zero.  Pairing with $z$
gives $z=0$.  For two right-hand sides, subtract the two fixed-point
identities and use \eqref{eq:map-contraction}; the contraction term is
absorbed to obtain
\begin{equation}
 \norm{\psi_1-\psi_2}_{\cG^s_\eps}
 \le 2C_L\norm{f_1-f_2}_{H^s}. 
 \label{eq:poisson-Lipschitz}
\end{equation}
This avoids requiring $W^{s,\infty}$ regularity of the nonlinear screening
coefficient in \eqref{eq:poisson-uniqueness}.

Finally differentiate \eqref{eq:nonlinear-poisson}:
\begin{equation}
 \left(-\eps^2\Delta+e^{\phi^a+\psi}\right)\pa_t\psi
 =\pa_tr+\pa_t\cR_\phi^a
 -(\pa_t\phi^a)e^{\phi^a}(e^\psi-1). \label{eq:poisson-time-equation}
\end{equation}
To retain the smooth profile coefficient in the linear inverse, rewrite this
equation as
\begin{equation}
 L_\eps\pa_t\psi
 =\pa_tr+\pa_t\cR_\phi^a
 -(\pa_t\phi^a)e^{\phi^a}(e^\psi-1)
 -e^{\phi^a}(e^\psi-1)\pa_t\psi. 
 \label{eq:poisson-time-fixed-coefficient}
\end{equation}
Because $s-1\ge2$, the $H^{s-1}$ product estimate and
\eqref{eq:poisson-estimate} imply
\[
 \norm{e^{\phi^a}(e^\psi-1)\pa_t\psi}_{H^{s-1}}
 \le C\norm\psi_{H^s}\norm{\pa_t\psi}_{H^{s-1}}.
\]
Apply \eqref{eq:linear-screened-estimate} to
\eqref{eq:poisson-time-fixed-coefficient}.  After reducing $\eta_P$ if
necessary, the last display is absorbed on the left.  The remaining profile
term is bounded by $C\norm\psi_{H^s}$, and
\eqref{eq:poisson-estimate} gives \eqref{eq:poisson-time}.  At no point is an
unweighted derivative of order $s+1$ estimated, so no division by $\eps$
occurs.
\end{proof}

\begin{corollary}[Compatible initial potential]\label{cor:initial-potential}
Let $n_0=n^a(t_0)+r_0$ with $r_0\in H^s$, $\norm{r_0}_{H^s}$ sufficiently
small, and $n_0>0$.  Then there is a unique
$\phi_0=\phi^a(t_0)+\psi_0$ satisfying the Poisson equation and the prescribed
far fields.  Moreover,
\begin{equation}
  \norm{\psi_0}_{\cG^s_\eps}
  \le C\left(\norm{r_0}_{H^s}
  +\norm{\cR_\phi^a(t_0)}_{H^s}\right). \label{eq:initial-potential-bound}
\end{equation}
Thus $r_0$ and $v_0$ are the freely chosen perturbative data; $\psi_0$ is a
constraint variable.
\end{corollary}

\begin{corollary}[Nonempty prepared-data class]\label{cor:prepared-data}
Take $n_0=n^a(t_0)$ and $u_0=u^a(t_0)$, and let $\phi_0$ be the compatible
potential.  Then
\begin{equation}
  \norm{\phi_0-\phi^a(t_0)}_{\cG^s_\eps}
  \le C\eps^{2M+2},\qquad
  \cE_{s,\eps}(t_0)\le C\eps^{4M+4}. \label{eq:canonical-prepared-data}
\end{equation}
The same conclusion holds after adding $H^s$ density and velocity
perturbations of size $O(\eps^{2M+2})$.
\end{corollary}

\begin{proof}
Set $r_0=v_0=0$ in \eqref{eq:initial-potential-bound} and use
\eqref{eq:residual-bound}.  Energy equivalence gives the second estimate.
The Lipschitz part of \Cref{prop:poisson-resolver} handles the stated
perturbations.
\end{proof}

\begin{lemma}[Commuted elliptic relation]\label{lem:commuted-poisson}
For $1\le\abs\alpha\le s$, set $A=e^{\phi^a+\psi}$.  Then
\begin{equation}
  \pa^\alpha r=-\eps^2\Delta\pa^\alpha\psi
  +A\pa^\alpha\psi+\mathcal C_\alpha-\pa^\alpha\cR_\phi^a, \label{eq:commuted-poisson}
\end{equation}
where
\begin{equation}
  \norm{\mathcal C_\alpha}_{2}
  \le C K_a\norm\psi_{H^{\abs\alpha-1}}
  +C\norm\psi_{H^{s-1}}\norm\psi_{H^s}. \label{eq:C-alpha}
\end{equation}
In addition,
\begin{align}
 \norm{\nabla\mathcal C_\alpha}_2
 &\le CK_a\norm\psi_{H^{\abs\alpha}}
 +C\norm\psi_{H^s}^2, \label{eq:grad-C-alpha}\\
 \sum_{1\le\abs\alpha\le s}\norm{\pa_t\mathcal C_\alpha}_2
 &\le C(K_a+\norm\psi_{H^s})
 \left(\norm\psi_{H^s}+\norm{\pa_t\psi}_{H^{s-1}}\right).
 \label{eq:time-C-alpha}
\end{align}
Thus the time-differentiated relation contains no unweighted top derivative
of $\nabla\psi$.
\end{lemma}

\begin{proof}
Set $G(\phi^a,\psi)=e^{\phi^a}(e^\psi-1)$.  The multivariable chain rule
gives
\begin{equation}
 \pa^\alpha G=A\pa^\alpha\psi+\mathcal C_\alpha,
 \qquad
 \mathcal C_\alpha=
 \sum c_{\mathbf\beta,\mathbf\gamma}
 e^{\phi^a+\vartheta\psi}
 \prod_{j=1}^J\pa^{\beta_j}\phi^a
 \prod_{k=1}^K\pa^{\gamma_k}\psi,
 \label{eq:C-alpha-expansion}
\end{equation}
where $0\le\vartheta\le1$ is represented by an integral remainder,
$\sum_j\abs{\beta_j}+\sum_k\abs{\gamma_k}=\abs\alpha$, and the sole term
with $K=1$, $\gamma_1=\alpha$, and no differentiated profile factor has been
removed.  Thus each remaining product either contains a derivative of
$\phi^a$ or at least two factors involving $\psi$, and no single
$\psi$-factor carries more than $\abs\alpha-1$ derivatives.

Place the factor with the largest number of derivatives in $L^2$ and all
others in $L^\infty$.  Since $s\ge5$ in two dimensions,
$H^{s-2}\hookrightarrow L^\infty$, and the two largest derivative counts in
a product cannot both exceed $s-2$.  This gives \eqref{eq:C-alpha}.  Applying
one additional spatial derivative to \eqref{eq:C-alpha-expansion} either
puts at most $\abs\alpha$ derivatives on one $\psi$-factor or differentiates
a profile coefficient; this proves \eqref{eq:grad-C-alpha}.

Finally, differentiate \eqref{eq:C-alpha-expansion} in time.  If $\pa_t$
falls on a $\psi$-factor, that factor originally carries at most
$\abs\alpha-1\le s-1$ spatial derivatives; if it falls on the exponential,
the new factor is $\pa_t\phi^a+\pa_t\psi$.  The same $L^2$--$L^\infty$
allocation and the algebra property of $H^{s-1}$ give
\eqref{eq:time-C-alpha}.  Estimate \eqref{eq:poisson-time} supplies the
required time norm and completes the proof.
\end{proof}

\subsection{Perturbation equations and the basic energy}\label{sec:energy}

Subtracting the profile equations gives
\begin{align}
  \pa_tr+u\cdot\nabla r+n\Div v
  &=-v\cdot\nabla n^a-r\Div u^a-\cR_n^a,
  \label{eq:pert-mass}\\
  \pa_tv+u\cdot\nabla v+h_i'(n)\nabla r+\nabla\psi
  &=-v\cdot\nabla u^a
  -[h_i'(n)-h_i'(n^a)]\nabla n^a-\cR_u^a.
  \label{eq:pert-momentum}
\end{align}
This form retains the full transport velocity and the full acoustic
coefficient on the left.  In particular, no term of the form
$r\nabla r$ is treated as an $H^s$ source.  The only nonlinear pressure term
on the right contains the fixed profile gradient and obeys
\begin{equation}
 \norm{[h_i'(n)-h_i'(n^a)]\nabla n^a}_{H^s}
 \le CK_a(t)\norm r_{H^s}. \label{eq:profile-pressure-source}
\end{equation}
We shall also use the exactly equivalent conservative identity
\begin{equation}
 \pa_tr+\Div(nv+r u^a)=-\cR_n^a. 
 \label{eq:pert-mass-conservative}
\end{equation}
It follows directly by subtracting
$\pa_tn^a+\Div(n^au^a)=\cR_n^a$ from the full continuity equation.

At zeroth order the natural electrostatic term is the Legendre counterpart
$\cH_e^*$ defined in \eqref{eq:dual-electron-setting}, rather than
$e^{\phi^a}(e^\psi-1-\psi)$.
It is nonnegative and equivalent to $\psi^2$ for small $\psi$.  Indeed,
\[
  \frac{\dd}{\dd\psi}\bigl[e^\psi(\psi-1)+1\bigr]=e^\psi\psi.
\]

\begin{lemma}[Electric cancellation at zeroth order]\label{lem:zero-electric}
Assume the bootstrap bounds of \Cref{sec:high-order}.  The electric work in
the fluid perturbation energy satisfies
\begin{equation}
\begin{split}
  -\int_\Om n v\cdot\nabla\psi\,\dd x
  =-&\frac{\dd}{\dd t}\int_\Om
  \left\{\frac{\eps^2}{2}\abs{\nabla\psi}^2
  +\cH_e^*(\psi;\phi^a)\right\}\,\dd x
  +\mathfrak e_0, \label{eq:zero-electric}
\end{split}
\end{equation}
where
\begin{equation}
  \abs{\mathfrak e_0}
  \le C\bigl(K_a(t)+\norm W_{H^s}\bigr)\cE_{s,\eps}
  +C\bigl(\mathfrak R_s^a\bigr)^2. \label{eq:e0}
\end{equation}
Here $K_a$ is a bounded combination of profile derivatives.
\end{lemma}

\begin{proof}
Integration by parts gives
\[
  -\int nv\cdot\nabla\psi=\int\Div(nv)\psi.
\]
Use \eqref{eq:pert-mass-conservative} to write
\begin{equation}
 \Div(nv)=-\pa_tr-u^a\cdot\nabla r-r\Div u^a-\cR_n^a.
 \label{eq:div-nv}
\end{equation}
The two profile-transport terms combine without estimating
$\nabla\psi$:
\begin{equation}
 \int(-u^a\cdot\nabla r-r\Div u^a)\psi
 =\int r u^a\cdot\nabla\psi. 
 \label{eq:zero-transport-first}
\end{equation}
Substitute the undifferentiated Poisson equation
\[
 r=-\eps^2\Delta\psi+e^{\phi^a}(e^\psi-1)-\cR_\phi^a.
\]
For the gradient part, integration by parts gives
\begin{align}
 -\eps^2\int (u^a\cdot\nabla\psi)\Delta\psi
 ={}&\eps^2\int \pa_k u_j^a\,\pa_j\psi\,\pa_k\psi
 -\frac{\eps^2}{2}\int(\Div u^a)\abs{\nabla\psi}^2. 
 \label{eq:zero-transport-gradient}
\end{align}
With $F(\psi)=e^\psi-1-\psi$, the screened part is
\begin{align}
 \int e^{\phi^a}(e^\psi-1)u^a\cdot\nabla\psi
 =-\int\Div(e^{\phi^a}u^a)F(\psi), 
 \label{eq:zero-transport-screened}
\end{align}
and the residual part is
\begin{align}
 -\int \cR_\phi^a u^a\cdot\nabla\psi
 =\int\Div(\cR_\phi^a u^a)\psi. 
 \label{eq:zero-transport-residual}
\end{align}
Equations \eqref{eq:zero-transport-gradient}--
\eqref{eq:zero-transport-residual} are bounded by
$CK_a\cE_{s,\eps}+C(\mathfrak R_s^a)^2$ because
$\abs{F(\psi)}\le C\psi^2$ on the bootstrap range.  This is the step that
avoids an unweighted $L^2$ norm of $\nabla\psi$.

It remains to treat $-\int\pa_tr\,\psi$.  Differentiate
\eqref{eq:nonlinear-poisson} in time and multiply by $\psi$:
\begin{align*}
  \int \pa_tr\,\psi
  &=\eps^2\int\nabla\pa_t\psi\cdot\nabla\psi
  +\int e^{\phi^a+\psi}\pa_t\psi\,\psi
  +\int (\pa_te^{\phi^a})(e^\psi-1)\psi
  -\int\pa_t\cR_\phi^a\psi.
\end{align*}
If $h(\psi)=e^\psi(\psi-1)+1$, then $h'(\psi)=e^\psi\psi$ and
\begin{align}
 \int e^{\phi^a+\psi}\pa_t\psi\,\psi
 =\frac{\dd}{\dd t}\int e^{\phi^a}h(\psi)
 -\int(\pa_t\phi^a)e^{\phi^a}h(\psi). 
 \label{eq:zero-entropy-time}
\end{align}
The first gradient term and \eqref{eq:zero-entropy-time} give the derivative
in \eqref{eq:zero-electric}.  The difference between the two remaining
profile-coefficient terms is quadratic in $\psi$, while the residual terms
are bounded by Young's inequality using
\eqref{eq:poisson-estimate} and \eqref{eq:residual-norm}.  Together with
\eqref{eq:zero-transport-gradient}--\eqref{eq:zero-transport-residual}, this proves
\eqref{eq:e0} and the lemma.
\end{proof}

\subsection{High-order compensated estimate}\label{sec:high-order}

Set
\begin{equation}
  \mathfrak N_s(t)=\norm r_{H^s}+\norm v_{H^s}
  +\norm\psi_{H^s}+\eps\norm{\nabla\psi}_{H^s}. \label{eq:Ns}
\end{equation}
We also write
\begin{equation}
  K_a(t)=1+\norm{\pa_t(n^a,u^a,\phi^a)}_{W^{s,\infty}}
  +\norm{\nabla(n^a,u^a,\phi^a)}_{W^{s,\infty}}. \label{eq:Ka}
\end{equation}
On an interval $[t_0,\tau]$ assume
\begin{equation}
  \sup_{t_0\le t\le\tau}\mathfrak N_s(t)\le2\eta,
  \qquad n^a+r\in[n_*/4,2n^*]. \label{eq:bootstrap}
\end{equation}

\begin{lemma}[Energy equivalence]\label{lem:energy-equivalence}
For $\eta$ and $\eps_0$ sufficiently small,
\begin{equation}
  c_0\mathfrak N_s(t)^2
  \le\cE_{s,\eps}(t)
  \le C_0\mathfrak N_s(t)^2, \label{eq:energy-equivalence}
\end{equation}
with $c_0,C_0$ independent of $\eps$.
\end{lemma}

\begin{proof}
The profile construction gives
$3n_*/4\le n^a\le5n^*/4$ after reducing $\eps_0$.  Under
\eqref{eq:bootstrap}, $n=n^a+r$ lies in the fixed compact interval
$I^\sharp$.  Hence
\begin{equation}
 \frac{n_*}{4}\le n\le2n^*,\qquad
 \frac{p_*}{2n^*}\le h_i'(n)\le C_{I^\sharp},
 \label{eq:fluid-weight-bounds}
\end{equation}
where the upper bound uses the prescribed $C^1$ norm of $p_i$.
Sobolev embedding and $\mathfrak N_s\le2\eta$ give
$\norm\psi_\infty\le2C_S\eta$.  The profile potential remains in a
compact interval because $\phi^a=\log\Nr+O(\eps^2)$ uniformly.  Therefore
\begin{equation}
 0<a_0\le e^{\phi^a+\psi}\le a_1
 \label{eq:screening-weight-bounds}
\end{equation}
with constants independent of $\eps$.

For the undifferentiated electron energy set
$h(z)=e^z(z-1)+1$.  Since $h(0)=h'(0)=0$ and
$h''(z)=e^z(1+z)$, Taylor's formula with integral remainder yields
\begin{equation}
 h(z)=z^2\int_0^1(1-\theta)e^{\theta z}(1+\theta z)\,\dd\theta.
 \label{eq:electron-entropy-Taylor}
\end{equation}
Choose $\eta$ so that $2C_S\eta\le1/2$.  Then the integral in
\eqref{eq:electron-entropy-Taylor} is bounded above and below by positive
absolute constants.  Combining it with the compact bounds for $e^{\phi^a}$
gives
\begin{equation}
 c_e\psi^2\le\cH_e^*(\psi;\phi^a)\le C_e\psi^2.
 \label{eq:electron-entropy-equivalence}
\end{equation}
Summing \eqref{eq:fluid-weight-bounds},
\eqref{eq:screening-weight-bounds}, and
\eqref{eq:electron-entropy-equivalence} over $\abs\alpha\le s$ proves both
inequalities in \eqref{eq:energy-equivalence}.  Notice that every derivative
of order $s+1$ on $\psi$ is multiplied by $\eps$; the equivalence constants
therefore contain no hidden factor $\eps^{-1}$.
\end{proof}

The derivative count behind the compensated estimate is summarized below.
The entry ``available control'' records the norm used before Young's
inequality.
\begin{center}
\begin{tabular}{@{}lll@{}}
\toprule
Term & Reduction & Available control\\
\midrule
$\pa^\alpha\nabla\psi$, $\abs\alpha=s$
 & integrate the fluid work by parts
 & $\pa^\alpha\psi$ in $L^2$\\
$\pa_t\pa^\alpha r$
 & time-differentiated Poisson equation
 & $\pa^\alpha\psi$, $\eps\nabla\pa^\alpha\psi$\\
$\pa_t\mathcal C_\alpha$
 & $\abs\beta\le s-1$ on $\pa_t\pa^\beta\psi$
 & \eqref{eq:poisson-time}\\
profile and transport commutators
 & Moser estimates
 & $K_a\mathfrak N_s$\\
\bottomrule
\end{tabular}
\end{center}

\begin{lemma}[Time derivatives below top order]\label{lem:time-derivatives}
Under \eqref{eq:bootstrap},
\begin{equation}
 \norm{\pa_tr}_{H^{s-1}}+\norm{\pa_tv}_{H^{s-1}}
 +\norm{\pa_t\psi}_{H^{s-1}}
 +\eps\norm{\nabla\pa_t\psi}_{H^{s-1}}
 \le C(K_a+\mathfrak N_s)\mathfrak N_s
 +C\mathfrak R_s^a. \label{eq:time-derivative-estimate}
\end{equation}
The constant is independent of $\eps$.
\end{lemma}

\begin{proof}
Because $s-1\ge4$, $H^{s-1}(\Om)$ is an algebra and
$H^{s-1}\hookrightarrow W^{1,\infty}$.  From
\eqref{eq:pert-mass},
\begin{align}
 \norm{\pa_tr}_{H^{s-1}}
 \le{}&C(1+K_a+\norm v_{H^s})\norm r_{H^s}
 +C(1+K_a+\norm r_{H^s})\norm v_{H^s}\notag\\
 &\quad
 +CK_a(\norm r_{H^s}+\norm v_{H^s})
 +\norm{\cR_n^a}_{H^{s-1}}. 
 \label{eq:rt-expanded}
\end{align}
Here and below an undifferentiated profile coefficient is estimated in
$L^\infty$ and only its derivatives are placed in $H^s$; thus the distinct
far fields cause no infinite norm.

For the momentum equation, the composition estimate on the compact state
interval gives
\begin{align}
 \norm{h_i'(n)\nabla r}_{H^{s-1}}
 &\le C(1+K_a+\norm r_{H^s})\norm r_{H^s},\\
 \norm{[h_i'(n)-h_i'(n^a)]\nabla n^a}_{H^{s-1}}
 &\le CK_a\norm r_{H^{s-1}}.
 \label{eq:pressure-time-bound}
\end{align}
The electric term is harmless at this level because
$\norm{\nabla\psi}_{H^{s-1}}\le\norm\psi_{H^s}$; the unavailable quantity
would be $\norm{\nabla\psi}_{H^s}$, which is not used.  Consequently,
\begin{align}
 \norm{\pa_tv}_{H^{s-1}}
 \le C(K_a+\mathfrak N_s)\mathfrak N_s
 +C\norm{\cR_u^a}_{H^{s-1}}. 
 \label{eq:vt-expanded}
\end{align}
Equations \eqref{eq:rt-expanded}--\eqref{eq:vt-expanded} give the first two
terms in \eqref{eq:time-derivative-estimate}.  Insert
\eqref{eq:rt-expanded} into \eqref{eq:poisson-time}; after using
\eqref{eq:poisson-estimate} and the residual definition, the two potential
terms follow.  All constants depend only on the bootstrap state interval and
profile seminorms, not on $\eps$.
\end{proof}

The central calculation is the following lemma.

\begin{lemma}[Top-order electric compensation]\label{lem:electric-compensation}
For every spatial multi-index $1\le\abs\alpha\le s$, the leading electric work
in the $\pa^\alpha$ momentum energy satisfies
\begin{equation}
\begin{split}
  -\int_\Om n\pa^\alpha v\cdot\nabla\pa^\alpha\psi\,\dd x
  =-&\frac12\frac{\dd}{\dd t}\int_\Om
  \left(A\abs{\pa^\alpha\psi}^2
  +\eps^2\abs{\nabla\pa^\alpha\psi}^2\right)\,\dd x
  +\mathfrak e_\alpha, \label{eq:top-electric}
\end{split}
\end{equation}
where
\begin{equation}
  \sum_{1\le\abs\alpha\le s}\abs{\mathfrak e_\alpha}
  \le C\bigl(K_a(t)+\mathfrak N_s(t)\bigr)\cE_{s,\eps}(t)
  +C\bigl(\mathfrak R_s^a(t)\bigr)^2. \label{eq:ealpha}
\end{equation}
No term on the right contains $\eps^{-1}$ or an unweighted
$H^s$ norm of $\nabla\psi$.
\end{lemma}

\begin{proof}
Put $q_\alpha=\pa^\alpha\psi$ and
$\rho_\alpha=\pa^\alpha r$.  Integrating the electric work by parts gives
\begin{equation}
  -\int n\pa^\alpha v\cdot\nabla\pa^\alpha\psi
  =\int\Div(n\pa^\alpha v)q_\alpha. \label{eq:electric-ibp}
\end{equation}
Commuting the nonconservative identity \eqref{eq:pert-mass} and then adding
$\nabla n\cdot\pa^\alpha v$ gives
\begin{equation}
  \Div(n\pa^\alpha v)
  =-\pa_t\rho_\alpha-u\cdot\nabla\rho_\alpha
  +\mathcal M_\alpha, \label{eq:mass-top}
\end{equation}
where
\begin{equation}
\begin{split}
 \mathcal M_\alpha={}&-[\pa^\alpha,u]\cdot\nabla r
+\nabla n\cdot\pa^\alpha v-[\pa^\alpha,n]\Div v\\
 &- \pa^\alpha(v\cdot\nabla n^a)
-\pa^\alpha(r\Div u^a)-\pa^\alpha\cR_n^a.
\end{split}
\label{eq:Malpha-definition}
\end{equation}
The apparent top products cancel: the term
$-(\pa^\alpha v)\cdot\nabla r$ in the first commutator cancels the
$\nabla r\cdot\pa^\alpha v$ part of the second term, and the term
$-(\pa^\alpha v)\cdot\nabla n^a$ cancels its profile counterpart.  After
these cancellations, the standard commutator estimate gives
\begin{equation}
  \sum_{1\le\abs\alpha\le s}\norm{\mathcal M_\alpha}_2
  \le C(K_a+\mathfrak N_s)(\norm r_{H^s}+\norm v_{H^s})
  +C\norm{\cR_n^a}_{H^s}. \label{eq:Malpha}
\end{equation}

The transport term in \eqref{eq:mass-top} cannot be estimated by
$\norm{\nabla q_\alpha}_2$ without losing $\eps^{-1}$.  Instead, integrate
it once and retain $\rho_\alpha$:
\begin{equation}
 -\int (u\cdot\nabla\rho_\alpha)q_\alpha
 =\int \rho_\alpha u\cdot\nabla q_\alpha
 +\int(\Div u)\rho_\alpha q_\alpha. 
 \label{eq:top-transport-rho}
\end{equation}
Use \eqref{eq:commuted-poisson} in the first term.  Its three principal
pieces satisfy
\begin{align}
 -\eps^2\int (u\cdot\nabla q_\alpha)\Delta q_\alpha
 ={}&\eps^2\int\pa_ku_j\,\pa_jq_\alpha\,\pa_kq_\alpha
 -\frac{\eps^2}{2}\int(\Div u)\abs{\nabla q_\alpha}^2,
 \label{eq:top-transport-gradient}\\
 \int A q_\alpha u\cdot\nabla q_\alpha
 ={}&-\frac12\int\Div(Au)q_\alpha^2,
 \label{eq:top-transport-screening}\\
 \int u(\mathcal C_\alpha-\pa^\alpha\cR_\phi^a)
 \cdot\nabla q_\alpha
 ={}&-\int\Div\!\left[u(\mathcal C_\alpha
 -\pa^\alpha\cR_\phi^a)\right]q_\alpha.
 \label{eq:top-transport-commutator}
\end{align}
The last identity is legitimate because \eqref{eq:grad-C-alpha} controls
$\nabla\mathcal C_\alpha$ and the residual is available in $H^{s+1}$.
The second term in \eqref{eq:top-transport-rho} is bounded directly by
$\norm{\Div u}_\infty\norm{\rho_\alpha}_2\norm{q_\alpha}_2$.
Consequently,
\begin{equation}
 \sum_{1\le\abs\alpha\le s}
 \abs{-\ip{u\cdot\nabla\rho_\alpha}{q_\alpha}}
 \le C(K_a+\mathfrak N_s)\cE_{s,\eps}
 +C(\mathfrak R_s^a)^2. 
 \label{eq:top-transport-bound}
\end{equation}

We now extract the time derivative.  Differentiate
\eqref{eq:commuted-poisson} in time:
\begin{equation}
  \pa_t\rho_\alpha
  =-\eps^2\Delta\pa_tq_\alpha
  +A\pa_tq_\alpha
  +(\pa_tA)q_\alpha
  +\pa_t\mathcal C_\alpha-\pa_t\pa^\alpha\cR_\phi^a. \label{eq:time-commuted-poisson}
\end{equation}
Pairing with $q_\alpha$ yields the exact identity
\begin{align}
 \ip{\pa_t\rho_\alpha}{q_\alpha}
 ={}&\frac12\frac{\dd}{\dd t}\int
 \left(\eps^2\abs{\nabla q_\alpha}^2+Aq_\alpha^2\right)
 +\frac12\int(\pa_tA)q_\alpha^2\notag\\
 &+\ip{\pa_t\mathcal C_\alpha}{q_\alpha}
 -\ip{\pa_t\pa^\alpha\cR_\phi^a}{q_\alpha}.
 \label{eq:time-commuted-pairing}
\end{align}
Since
$\pa_tA=A(\pa_t\phi^a+\pa_t\psi)$,
\Cref{lem:time-derivatives} and Sobolev embedding give
$\norm{\pa_tA}_\infty
\le C(K_a+\mathfrak N_s)+C\mathfrak R_s^a$.
Furthermore, \eqref{eq:time-C-alpha} gives
\begin{equation}
  \sum_{1\le\abs\alpha\le s}
  \abs{\ip{\pa_t\mathcal C_\alpha}{\pa^\alpha\psi}}
  \le C(K_a+\mathfrak N_s)\cE_{s,\eps}
  +C\bigl(\mathfrak R_s^a\bigr)^2.
 \label{eq:time-C-pairing}
\end{equation}
The residual pairing in \eqref{eq:time-commuted-pairing} is treated by
Young's inequality.  Finally,
$\abs{\ip{\mathcal M_\alpha}{q_\alpha}}$ is controlled by
\eqref{eq:Malpha} and the $H^s$ part of the energy.  Insert
\eqref{eq:mass-top}, \eqref{eq:top-transport-bound}, and
\eqref{eq:time-commuted-pairing} into \eqref{eq:electric-ibp}; summing over
$\alpha$ proves \eqref{eq:top-electric}--\eqref{eq:ealpha}.  The only
top spatial derivative of the potential appears in
$\eps\nabla q_\alpha$, exactly with its energy weight.
\end{proof}

\begin{proposition}[Uniform high-order estimate]\label{prop:energy}
Under \eqref{eq:bootstrap},
\begin{equation}
  \frac{\dd}{\dd t}\cE_{s,\eps}(t)
  \le C\bigl(K_a(t)+\mathfrak N_s(t)\bigr)\cE_{s,\eps}(t)
  +C\bigl(\mathfrak R_s^a(t)\bigr)^2, \label{eq:differential-energy}
\end{equation}
where $C$ is independent of $\eps$ and $K_a$ is defined in \eqref{eq:Ka}.
Consequently,
\begin{equation}
  \sup_{t_0\le t\le\tau}\cE_{s,\eps}(t)
  \le C_{t_0,T}\left(
  \cE_{s,\eps}(t_0)+
  \int_{t_0}^{\tau}\bigl(\mathfrak R_s^a(t)\bigr)^2\,\dd t\right). \label{eq:integrated-energy}
\end{equation}
\end{proposition}

\begin{proof}
Let $a(n)=h_i'(n)$ and, for $\abs\alpha\le s$, write the commuted fluid
equations as
\begin{align}
 \pa_t\pa^\alpha r+u\cdot\nabla\pa^\alpha r
 +n\Div\pa^\alpha v&=F_\alpha^r, 
 \label{eq:commuted-fluid-mass}\\
 \pa_t\pa^\alpha v+u\cdot\nabla\pa^\alpha v
 +a(n)\nabla\pa^\alpha r+\nabla\pa^\alpha\psi
 &=F_\alpha^v, 
 \label{eq:commuted-fluid-momentum}
\end{align}
where
\begin{align}
 F_\alpha^r={}&-[\pa^\alpha,u]\cdot\nabla r
 -[\pa^\alpha,n]\Div v
 -\pa^\alpha(v\cdot\nabla n^a+r\Div u^a+\cR_n^a),
 \label{eq:Fr-alpha}\\
 F_\alpha^v={}&-[\pa^\alpha,u]\cdot\nabla v
 -[\pa^\alpha,a(n)]\nabla r
 -\pa^\alpha\left(v\cdot\nabla u^a
 +[a(n)-a(n^a)]\nabla n^a+\cR_u^a\right).
 \label{eq:Fv-alpha}
\end{align}
For $\alpha=0$, commutators are understood to vanish.

The tame estimates in two dimensions give
\begin{equation}
 \sum_{\abs\alpha\le s}
 \left(\norm{F_\alpha^r}_2+\norm{F_\alpha^v}_2\right)
 \le C(K_a+\mathfrak N_s)(\norm r_{H^s}+\norm v_{H^s})
 +C\mathfrak R_s^a. 
 \label{eq:fluid-source-bound}
\end{equation}
Indeed, for $1\le\abs\alpha\le s$,
\begin{align}
 \norm{[\pa^\alpha,u]\nabla f}_2
 &\le C\left(\norm{\nabla u}_\infty\norm f_{H^s}
 +\norm{\nabla u}_{H^{s-1}}\norm{\nabla f}_\infty\right),
 \label{eq:transport-commutator-full}\\
 \norm{[\pa^\alpha,a(n)]\nabla r}_2
 &\le C(1+K_a+\norm r_{H^s})\norm r_{H^s},
 \label{eq:pressure-commutator-full}
\end{align}
and the analogous estimate holds for $[\pa^\alpha,n]\Div v$.
The composition constants are uniform because $n$ remains in the fixed
compact interval $I^\sharp$.  The profile terms use one bounded derivative
of $(n^a,u^a)$ and the localized higher derivatives recorded by $K_a$.
This proves \eqref{eq:fluid-source-bound} without placing the nonconstant
profiles themselves in $L^2$.

Multiply \eqref{eq:commuted-fluid-mass} by
$a(n)\pa^\alpha r$ and
\eqref{eq:commuted-fluid-momentum} by
$n\pa^\alpha v$.  The transport terms are exact coefficient derivatives:
\begin{align}
 \int a(n)\pa^\alpha r
 (\pa_t+u\cdot\nabla)\pa^\alpha r
 ={}&\frac12\frac{\dd}{\dd t}\int a(n)\abs{\pa^\alpha r}^2
 -\frac12\int
 \bigl(\pa_ta(n)+\Div(a(n)u)\bigr)
 \abs{\pa^\alpha r}^2, 
 \label{eq:density-transport-energy}\\
 \int n\pa^\alpha v\cdot
 (\pa_t+u\cdot\nabla)\pa^\alpha v
 ={}&\frac12\frac{\dd}{\dd t}\int n\abs{\pa^\alpha v}^2.
 \label{eq:velocity-transport-energy}
\end{align}
The second identity uses the full continuity equation
$\pa_tn+\Div(nu)=0$.  In the first,
\[
 \pa_ta(n)+u\cdot\nabla a(n)
 =-n a'(n)\Div u,
\]
so its coefficient is bounded by $C(K_a+\mathfrak N_s)$.

Since $na(n)=p_i'(n)$, the two principal acoustic terms satisfy the exact
identity
\begin{equation}
 \int p_i'(n)\pa^\alpha r\,\Div\pa^\alpha v\,\dd x
 +\int p_i'(n)\pa^\alpha v\cdot\nabla\pa^\alpha r\,\dd x
 =-\int \nabla p_i'(n)\cdot\pa^\alpha v\,
 \pa^\alpha r\,\dd x. \label{eq:nonlinear-acoustic-cancellation}
\end{equation}
The right-hand side is bounded by
$C(K_a+\mathfrak N_s)\cE_{s,\eps}$ because
$\nabla p_i'(n)=p_i''(n)(\nabla n^a+\nabla r)$ is bounded in $L^\infty$.
The commuted sources satisfy, by \eqref{eq:fluid-source-bound},
\begin{align}
 \sum_{\abs\alpha\le s}
 \abs{\ip{a(n)F_\alpha^r}{\pa^\alpha r}
 +\ip{nF_\alpha^v}{\pa^\alpha v}}
 \le C(K_a+\mathfrak N_s)\cE_{s,\eps}
 +C(\mathfrak R_s^a)^2. 
 \label{eq:fluid-source-energy}
\end{align}
Here Young's inequality is used only for the residual factors.

At order zero, insert \Cref{lem:zero-electric}.  At orders
$1\le\abs\alpha\le s$, insert \Cref{lem:electric-compensation}.  The two
identities produce precisely the last two lines of
\eqref{eq:energy-definition}.  Summing proves
\eqref{eq:differential-energy}.  Under the bootstrap bound,
\begin{equation}
 \int_{t_0}^\tau(K_a(t)+\mathfrak N_s(t))\,\dd t
 \le C_{M,t_0,T}+2\eta(T-t_0). 
 \label{eq:coefficient-integral}
\end{equation}
Gronwall's inequality applied to \eqref{eq:differential-energy} therefore
gives \eqref{eq:integrated-energy}, with a constant depending on the fixed
profile and interval but not on $\eps$ or $\tau$.
\end{proof}

We explain why warm-ion pressure is used.
The screened part controls $\psi$, but not $r$ uniformly at high frequency.
The positive coefficient $h_i'(n)$ provides the structural mechanism to control 
the $\eps$-independent density energy. The cold-ion case requires a completely 
different topology and derivative hierarchy. 

\subsection{Existence, continuation, and proof of the main theorem}\label{sec:existence}

We give the construction because uniform continuation is part of the claim.
At each fixed time denote by $\mathcal P_\eps(t,r)$ the solution $\psi$
supplied by \Cref{prop:poisson-resolver}.  The dependence on $t$ enters only
through the prescribed composite profile and its residual.  The local
Lipschitz estimate and \eqref{eq:poisson-time-equation} show that
\begin{equation}
 \mathcal P_\eps(t,\cdot):\{r\in H^s:\norm r_{H^s}<\eta_P\}
 \longrightarrow\{\psi:\norm\psi_{\cG^s_\eps}<C\eta_P\}
 \label{eq:poisson-map}
\end{equation}
is smooth on a smaller ball, uniformly for $t\in[t_0,T]$.  More precisely,
implicit differentiation of \eqref{eq:nonlinear-poisson} gives
\begin{equation}
 D_r\mathcal P_\eps(t,r)f
 =\left(-\eps^2\Delta+e^{\phi^a+\mathcal P_\eps(t,r)}\right)^{-1}f.
 \label{eq:poisson-Frechet}
\end{equation}
Higher Fr\'echet derivatives are obtained by differentiating products of
the lower ones and satisfy fixed-$\eps$ tame estimates.  The graph estimates
for the first derivative, rather than the unweighted norm of
$\nabla\mathcal P_\eps$, are uniform in $\eps$.

For fixed $\eps>0$, ordinary elliptic regularity improves the graph estimate:
if $r\in H^s$, then $\mathcal P_\eps(t,r)\in H^{s+2}$, with a constant that
may depend on $\eps^{-1}$.  Consequently, after substituting
$\psi=\mathcal P_\eps(t,r)$, the fluid equations form a standard symmetrizable
quasilinear hyperbolic system with a smooth nonlocal source.

\begin{proposition}[Fixed-$\eps$ local theory and continuation]
\label{prop:local-theory}
Fix $0<\eps\le\eps_0$.  Compatible $H^s$ data in the positive state range
with $\norm{r_0}_{H^s}<\eta_P/2$ generate a unique maximal perturbative
solution on $[t_0,T_\eps^{\max})$.  Here maximality is taken in the Poisson
chart \eqref{eq:poisson-map} and in the positive density region.
If $T_\eps^{\max}<\infty$, then either
\begin{equation}
\begin{split}
 &\liminf_{t\uparrow T_\eps^{\max}}\inf_x n(t,x)=0,
 \quad\mathrm{or}\quad
 \limsup_{t\uparrow T_\eps^{\max}}
 \bigl(\norm{r(t)}_{H^s}+\norm{v(t)}_{H^s}\bigr)=\infty,\\
 &\hspace{38mm}\mathrm{or}\quad
 \limsup_{t\uparrow T_\eps^{\max}}\norm{r(t)}_{H^s}\ge\eta_P.
 \label{eq:continuation-criterion}
\end{split}
\end{equation}
The local lifespan in this proposition is allowed to depend on $\eps$.
\end{proposition}

\begin{proof}
We separate construction, compactness, and continuation.

\emph{Step 1: Friedrichs equations.}
Let $J_\nu$ be a self-adjoint Fourier multiplier on $\Om$ with symbol
$\chi(\xi/\nu)$, where $\chi$ is even, smooth, compactly supported, and
equal to one near the origin.  Put
$Z_\nu=(r_\nu,v_\nu)$,
$n_\nu=n^a+r_\nu$, $u_\nu=u^a+v_\nu$, and
\begin{equation}
 \psi_\nu=\mathcal P_\eps(t,r_\nu).
 \label{eq:Friedrichs-potential}
\end{equation}
On the range of $J_\nu$, solve
\begin{align}
 \pa_tr_\nu={}&-J_\nu\bigl[
 u_\nu\cdot\nabla J_\nu r_\nu
 +n_\nu\Div J_\nu v_\nu
 +v_\nu\cdot\nabla n^a+r_\nu\Div u^a+\cR_n^a\bigr],
 \label{eq:Friedrichs-mass}\\
 \pa_tv_\nu={}&-J_\nu\bigl[
 u_\nu\cdot\nabla J_\nu v_\nu
 +h_i'(n_\nu)\nabla J_\nu r_\nu+
 \nabla\psi_\nu+v_\nu\cdot\nabla u^a\notag\\
 &\hspace{31mm}
 +[h_i'(n_\nu)-h_i'(n^a)]\nabla n^a+\cR_u^a\bigr],
 \label{eq:Friedrichs-momentum}
\end{align}
with $Z_\nu(t_0)=J_\nu Z_0$.  One may place $J_\nu$ on every occurrence of
$Z_\nu$ without changing the solution because the vector field takes values
in the band-limited subspace.  For fixed $\eps$, elliptic regularity and
\eqref{eq:poisson-Frechet} imply
\begin{equation}
 \norm{\mathcal P_\eps(t,r)}_{H^{s+2}}
 \le C_\eps(1+\norm r_{H^s}),
 \qquad
 \norm{D_r\mathcal P_\eps(t,r)f}_{H^{s+2}}
 \le C_\eps\norm f_{H^s}
 \label{eq:fixed-epsilon-elliptic}
\end{equation}
on compact subsets of the chart.  Thus the right-hand side of
\eqref{eq:Friedrichs-mass}--\eqref{eq:Friedrichs-momentum} is locally
Lipschitz on the band-limited $H^s$ space, and Picard's theorem gives a
smooth solution up to its first exit from that subset.

\emph{Step 2: bounds independent of the cutoff.}
The symmetric-hyperbolic calculation at fixed $\eps$ gives, while
$\norm{Z_\nu}_{H^s}\le R$, $n_\nu\ge c>0$, and
$\norm{r_\nu}_{H^s}\le3\eta_P/4$,
\begin{equation}
 \frac{\dd}{\dd t}\norm{Z_\nu}_{H^s}^2
 \le C_{\eps,R,c}\bigl(1+\norm{Z_\nu}_{H^s}^2\bigr).
 \label{eq:Friedrichs-fixed-epsilon-energy}
\end{equation}
The constant is independent of $\nu$: $J_\nu$ is self-adjoint, commutes
with derivatives, is bounded on every Sobolev space, and its commutator
with a variable coefficient is controlled by
\eqref{eq:cutoff-commutator}.  Taking a time
$\tau_\eps>t_0$ sufficiently close to $t_0$, the ODE estimate, Sobolev
embedding, and the initial margins give
\begin{equation}
 \sup_{t_0\le t\le\tau_\eps}\norm{Z_\nu(t)}_{H^s}\le2\norm{Z_0}_{H^s}+1,
 \quad n_\nu\ge\tfrac12\inf n_0,
 \quad\norm{r_\nu}_{H^s}<\tfrac34\eta_P.
 \label{eq:Friedrichs-common-time}
\end{equation}
This is a common time for all sufficiently large $\nu$.  Low-frequency
approximants can be enlarged to the same interval by shortening it once.

\emph{Step 3: the Cauchy limit.}
Because $Z_\nu$ remains in the range of $J_\nu$, the equations may be
written
\begin{equation}
 \pa_tZ_\nu=J_\nu\mathcal Q_\eps(t,Z_\nu),
 \label{eq:Friedrichs-vector-field}
\end{equation}
where $\mathcal Q_\eps$ is the unregularized fluid right-hand side after
substitution of $\psi=\mathcal P_\eps(t,r)$.  The single spatial derivative
in the hyperbolic part is important here: \eqref{eq:Friedrichs-common-time}
and the fixed-$\eps$ elliptic estimate give only
\begin{equation}
 \sup_\nu\norm{\mathcal Q_\eps(t,Z_\nu)}_{H^{s-1}}
 \le C_\eps.
 \label{eq:Friedrichs-vector-field-bound}
\end{equation}
Thus the Cauchy estimate must be made in $H^{s-2}$, not in $H^{s-1}$.

Put $W_{\nu\mu}=Z_\nu-Z_\mu$.  From
\[
 J_\nu\mathcal Q_\eps(t,Z_\nu)
 -J_\mu\mathcal Q_\eps(t,Z_\mu)
 =\mathcal Q_\eps(t,Z_\nu)-\mathcal Q_\eps(t,Z_\mu)
 +(J_\nu-I)\mathcal Q_\eps(t,Z_\nu)
 -(J_\mu-I)\mathcal Q_\eps(t,Z_\mu),
\]
the principal part of the difference is the unregularized symmetric
hyperbolic operator.  Commuting through order $s-2$, integrating that
principal part by parts, and estimating the coefficient differences by the
tame product bounds gives
\begin{equation}
 \frac{\dd}{\dd t}\norm{W_{\nu\mu}}_{H^{s-2}}^2
 \le C_\eps\norm{W_{\nu\mu}}_{H^{s-2}}^2
 +C\sum_{\kappa\in\{\nu,\mu\}}
 \norm{(J_\kappa-I)
        \mathcal Q_\eps(t,Z_\kappa)}_{H^{s-2}}^2.
 \label{eq:Friedrichs-difference-differential}
\end{equation}
Here coefficient differences are controlled in $H^{s-2}$ by the uniform
$H^s$ bound and the tame product estimates.  For the electric coefficient,
subtracting the two screened equations and applying fixed-$\eps$ elliptic
regularity at order $s-2$ yields
\begin{equation}
 \norm{\nabla(\psi_\nu-\psi_\mu)}_{H^{s-2}}
 \le C_\eps\norm{r_\nu-r_\mu}_{H^{s-2}}.
 \label{eq:approximate-potential-difference}
\end{equation}
The constant may depend on $\eps^{-1}$, which is harmless in this
fixed-$\eps$ construction.  By \Cref{lem:appendix-cutoff} and
\eqref{eq:Friedrichs-vector-field-bound},
\begin{equation}
 \sum_{\kappa\in\{\nu,\mu\}}
 \norm{(J_\kappa-I)\mathcal Q_\eps(t,Z_\kappa)}_{H^{s-2}}
 \le C_\eps(\nu^{-1}+\mu^{-1}).
 \label{eq:cutoff-error}
\end{equation}
After integration, the difference inequality is
\begin{equation}
 \sup_{t_0\le t\le\tau_\eps}
 \norm{Z_\nu(t)-Z_\mu(t)}_{H^{s-2}}^2
 \le C_\eps\left(
 \norm{(J_\nu-J_\mu)Z_0}_{H^{s-2}}^2
 +(\nu^{-1}+\mu^{-1})^2\right).
 \label{eq:Friedrichs-Cauchy}
\end{equation}
The right-hand side tends to zero because $Z_0\in H^s$.  Hence
$Z_\nu\to Z$ strongly in $C([t_0,\tau_\eps];H^{s-2})$ and weak-star in
$L^\infty([t_0,\tau_\eps];H^s)$.  Interpolation between these two bounds
gives, in particular,
\begin{equation}
 Z_\nu\longrightarrow Z
 \quad\hbox{in }C([t_0,\tau_\eps];H^{s'})
 \quad\text{for every }s'<s,
 \label{eq:Friedrichs-interpolated-convergence}
\end{equation}
including $s'=s-1$.  Choosing $s'>2$ is enough to pass all products in the
fluid equations.  The corresponding fixed-$\eps$ Lipschitz estimate for
$\mathcal P_\eps$, together with
\eqref{eq:Friedrichs-interpolated-convergence}, passes the nonlinear Poisson
constraint and the electric term.

It remains to recover strong continuity at the top index without assuming
top-order convergence of the Friedrichs sequence.  The limit equations and
the bounds just obtained imply
\[
 Z\in L^\infty H^s\cap C H^{s-1},\qquad
 \pa_tZ\in L^\infty H^{s-1},
\]
and therefore $Z\in C_w H^s$.  Apply a spatial mollifier $K_\delta$ to the
limit equations, commute $\pa^\alpha$, $\abs\alpha\le s$, through the
coefficients, and take the scalar product with the standard positive
symmetrizer.  The Moser estimates in \Cref{app:moser}, the uniform $H^s$
bound, and $Z\in C H^{s-1}$ show that every commutator is uniformly
integrable in time and converges in $L^1$ as $\delta\downarrow0$.  Thus the
limit itself satisfies, on every $[t_1,t_2]\subset[t_0,\tau_\eps]$, the
integral $H^s$ energy identity
\begin{equation}
 \mathfrak E_s(t_2)-\mathfrak E_s(t_1)
 =\int_{t_1}^{t_2}\mathfrak G_s(\tau)\,\dd\tau,
 \qquad \mathfrak G_s\in L^1(t_0,\tau_\eps).
 \label{eq:limit-top-energy-identity}
\end{equation}
In particular $\mathfrak E_s$ is continuous.  At a fixed time $t_1$, freeze
the positive symmetrizer at $Z(t_1)$.  Its difference from the symmetrizer
at time $t$ tends to zero in $L^\infty$ as $t\to t_1$, by the
$C H^{s-1}$ convergence.  Consequently
\eqref{eq:limit-top-energy-identity} gives convergence of the norms in that
fixed weighted Hilbert space.  Weak $H^s$ continuity and the Hilbert-space
fact ``weak convergence plus convergence of norms implies strong
convergence'' yield
\begin{equation}
 Z\in C([t_0,\tau_\eps];H^s).
 \label{eq:top-strong-continuity}
\end{equation}
Since the right-hand side of the limit system is then continuous in
$H^{s-1}$, the equations give $Z\in C^1H^{s-1}$;
\eqref{eq:poisson-time} gives the asserted time regularity of $\psi$.

\emph{Step 4: uniqueness and restart.}
The difference estimate in \Cref{lem:difference} below, used on the common
existence interval, gives uniqueness and continuous dependence.  The
construction can be restarted from any time $t_1$ at which
\begin{equation}
 \inf_x n(t_1,x)>0,\qquad
 \norm{Z(t_1)}_{H^s}<\infty,\qquad
 \norm{r(t_1)}_{H^s}<\eta_P.
 \label{eq:restart-conditions}
\end{equation}
Indeed, these strict conditions persist for a short time by Sobolev and
$H^s$ continuity, and \eqref{eq:fixed-epsilon-elliptic} supplies the
fixed-$\eps$ potential bound.  If none of the alternatives in
\eqref{eq:continuation-criterion} occurs, a sequence
$t_j\uparrow T_\eps^{\max}$ has uniform margins in
\eqref{eq:restart-conditions}; the local construction then continues the
solution beyond $T_\eps^{\max}$, a contradiction.  This proves the stated
criterion.
\end{proof}

\begin{lemma}[Difference estimate]\label{lem:difference}
Let two solutions remain in the bootstrap set and have the same composite
profile.  Put
\[
 \dot r=r^{(1)}-r^{(2)},\qquad
 \dot v=v^{(1)}-v^{(2)},\qquad
 \dot\psi=\psi^{(1)}-\psi^{(2)}
\]
and
\begin{equation}
 B=e^{\phi^a}\int_0^1
 e^{\theta\psi^{(1)}+(1-\theta)\psi^{(2)}}\,\dd\theta.
 \label{eq:difference-screening}
\end{equation}
Define
\begin{equation}
 \cE_{s-1,\eps}^{\rm diff}
 =\frac12\sum_{\abs\alpha\le s-1}\int_\Om
 \bigl(n^{(1)}\abs{\pa^\alpha\dot v}^2
 +h_i'(n^{(1)})\abs{\pa^\alpha\dot r}^2
 +B\abs{\pa^\alpha\dot\psi}^2
 +\eps^2\abs{\nabla\pa^\alpha\dot\psi}^2\bigr)\,\dd x.
 \label{eq:difference-energy}
\end{equation}
Then
\begin{equation}
  \frac{\dd}{\dd t}\cE_{s-1,\eps}^{\rm diff}
  \le C\bigl(K_a+\mathfrak N_s^{(1)}+\mathfrak N_s^{(2)}\bigr)
  \cE_{s-1,\eps}^{\rm diff}. \label{eq:difference}
\end{equation}
The constant is independent of $\eps$.
\end{lemma}

\begin{proof}
Set $n_i=n^a+r^{(i)}$, $u_i=u^a+v^{(i)}$, and
$a_i=h_i'(n_i)$.  Subtracting the perturbation equations and placing the
coefficients of the first solution on the left gives
\begin{align}
 \pa_t\dot r+u_1\cdot\nabla\dot r+n_1\Div\dot v
 ={}&-\dot v\cdot\nabla r^{(2)}-\dot r\Div v^{(2)}
 -\dot v\cdot\nabla n^a-\dot r\Div u^a,
 \label{eq:difference-fluid-mass}\\
 \pa_t\dot v+u_1\cdot\nabla\dot v+a_1\nabla\dot r
 +\nabla\dot\psi
 ={}&-\dot v\cdot\nabla(v^{(2)}+u^a)
 -(a_1-a_2)\nabla n_2.
 \label{eq:difference-fluid-momentum}
\end{align}
The difference of the nonlinear Poisson equations is the exact relation
\begin{equation}
 -\eps^2\Delta\dot\psi+B\dot\psi=\dot r.
 \label{eq:difference-poisson}
\end{equation}
The mean-value coefficient $B$ is bounded above and below uniformly, because
both solutions remain in the bootstrap set.  Put
\begin{equation}
 \Gamma(t)=1+K_a(t)+\mathfrak N_s^{(1)}(t)
 +\mathfrak N_s^{(2)}(t).
 \label{eq:difference-Gamma}
\end{equation}
The tame difference form of the composition estimate gives
\begin{equation}
 \norm{a_1-a_2}_{H^{s-1}}
 \le C\norm{\dot r}_{H^{s-1}},
 \qquad
 \norm B_{L^\infty}
 +\norm{\nabla B}_{H^{s-1}}+\norm{\pa_tB}_{H^{s-1}}
 \le C\Gamma.
 \label{eq:difference-coefficients}
\end{equation}
The coefficient itself is measured in $L^\infty$, while its spatial
derivatives are measured in $H^{s-1}$, as in
\Cref{def:profile-seminorm}.

For $\abs\alpha\le s-1$, let
$q_\alpha=\pa^\alpha\dot\psi$ and
$\rho_\alpha=\pa^\alpha\dot r$.  Commuting
\eqref{eq:difference-poisson} yields
\begin{equation}
 \rho_\alpha=-\eps^2\Delta q_\alpha+Bq_\alpha
 +\mathcal D_\alpha,
 \qquad
 \mathcal D_\alpha=[\pa^\alpha,B]\dot\psi.
 \label{eq:commuted-difference-poisson}
\end{equation}
The derivative in the commutator always falls at least once on $B$.
Consequently, the product estimates of \Cref{app:moser} give
\begin{equation}
 \sum_{\abs\alpha\le s-1}
 \left(\norm{\mathcal D_\alpha}_2+
 \norm{\nabla\mathcal D_\alpha}_2\right)
 \le C\Gamma\norm{\dot\psi}_{H^{s-1}}.
 \label{eq:difference-elliptic-commutator}
\end{equation}
The gradient estimate uses at most $s$ derivatives of $B$ and at most
$s-1$ derivatives of $\dot\psi$, which explains why the difference energy
is taken one order below the individual energies.

We record the corresponding time estimate.  Equations
\eqref{eq:difference-fluid-mass}--\eqref{eq:difference-fluid-momentum} imply
\begin{equation}
 \norm{\pa_t\dot r}_{H^{s-2}}+\norm{\pa_t\dot v}_{H^{s-2}}
 \le C\Gamma
 \left(\norm{\dot r}_{H^{s-1}}+\norm{\dot v}_{H^{s-1}}
 +\norm{\dot\psi}_{H^{s-1}}\right).
 \label{eq:difference-fluid-time}
\end{equation}
Differentiating \eqref{eq:difference-poisson} in time gives
\begin{equation}
 (-\eps^2\Delta+B)\pa_t\dot\psi
 =\pa_t\dot r-(\pa_tB)\dot\psi.
 \label{eq:difference-poisson-time}
\end{equation}
The two individual instances of \Cref{lem:time-derivatives}, together with
the difference Moser estimate, bound $\pa_tB$ in the relative coefficient
norms through order $s-1$.  Applying the screened estimate at order $s-2$ to
\eqref{eq:difference-poisson-time} therefore yields
\begin{equation}
 \norm{\pa_t\dot\psi}_{H^{s-2}}
 +\eps\norm{\nabla\pa_t\dot\psi}_{H^{s-2}}
 \le C\Gamma(\cE_{s-1,\eps}^{\rm diff})^{1/2}.
 \label{eq:difference-potential-time}
\end{equation}
Expanding $\pa_t\mathcal D_\alpha$ by Leibniz' rule, the factor
$\pa_t\dot\psi$ carries at most $s-2$ spatial derivatives, while the factor
$\pa_tB$ is paired with a lower derivative of $\dot\psi$.  Thus
\begin{equation}
 \sum_{\abs\alpha\le s-1}
 \norm{\pa_t\mathcal D_\alpha}_2
 \le C\Gamma(\cE_{s-1,\eps}^{\rm diff})^{1/2}.
 \label{eq:difference-commutator-time}
\end{equation}

We next perform the top electric calculation explicitly.  Commuting
\eqref{eq:difference-fluid-mass} gives
\begin{equation}
 \pa_t\rho_\alpha+u_1\cdot\nabla\rho_\alpha
 +n_1\Div\pa^\alpha\dot v=G_\alpha,
 \label{eq:commuted-difference-mass}
\end{equation}
where, by Leibniz' rule and \eqref{eq:difference-coefficients},
\begin{equation}
 \sum_{\abs\alpha\le s-1}\norm{G_\alpha}_2
 \le C\Gamma
 \left(\norm{\dot r}_{H^{s-1}}+\norm{\dot v}_{H^{s-1}}\right).
 \label{eq:difference-mass-source}
\end{equation}
It follows that
\begin{equation}
 \Div(n_1\pa^\alpha\dot v)
 =-\pa_t\rho_\alpha-u_1\cdot\nabla\rho_\alpha
 +G_\alpha+\nabla n_1\cdot\pa^\alpha\dot v.
 \label{eq:difference-divergence}
\end{equation}
After integration by parts, the electric work in the commuted momentum
equation is the pairing of the right-hand side of
\eqref{eq:difference-divergence} with $q_\alpha$.

Differentiate \eqref{eq:commuted-difference-poisson} in time and pair with
$q_\alpha$.  One obtains
\begin{equation}
 -\ip{\pa_t\rho_\alpha}{q_\alpha}
 ={}-\frac12\frac{\dd}{\dd t}\int
 \left(\eps^2\abs{\nabla q_\alpha}^2+Bq_\alpha^2\right)\,\dd x
 -\frac12\int(\pa_tB)q_\alpha^2\,\dd x
 -\ip{\pa_t\mathcal D_\alpha}{q_\alpha}.
 \label{eq:difference-electric-time}
\end{equation}
For the transport pairing, integrate once before inserting the elliptic
equation:
\begin{equation}
 -\ip{u_1\cdot\nabla\rho_\alpha}{q_\alpha}
 =\int\rho_\alpha u_1\cdot\nabla q_\alpha
 +\int(\Div u_1)\rho_\alpha q_\alpha.
 \label{eq:difference-electric-transport}
\end{equation}
The first term on the right is handled with
\eqref{eq:commuted-difference-poisson}.  The gradient and screened pieces
satisfy
\begin{align}
 -\eps^2\int(u_1\cdot\nabla q_\alpha)\Delta q_\alpha
 ={}&\eps^2\int\pa_ku_{1j}\,\pa_jq_\alpha\pa_kq_\alpha
 -\frac{\eps^2}{2}\int(\Div u_1)\abs{\nabla q_\alpha}^2,
 \label{eq:difference-transport-gradient}\\
 \int Bq_\alpha u_1\cdot\nabla q_\alpha
 ={}&-\frac12\int\Div(Bu_1)q_\alpha^2.
 \label{eq:difference-transport-screening}
\end{align}
For the commutator piece, use
\begin{equation}
 \int\mathcal D_\alpha u_1\cdot\nabla q_\alpha
 =-\int\Div(u_1\mathcal D_\alpha)q_\alpha,
 \label{eq:difference-transport-commutator}
\end{equation}
which is controlled by
\eqref{eq:difference-elliptic-commutator}.  The second term in
\eqref{eq:difference-electric-transport} is bounded after substituting
\eqref{eq:commuted-difference-poisson}; the Laplacian contribution is
integrated by parts once and is controlled by the weighted gradient energy.
Equations \eqref{eq:difference-electric-time}--
\eqref{eq:difference-transport-commutator}, together with
\eqref{eq:difference-mass-source} and
\eqref{eq:difference-commutator-time}, therefore show
\begin{equation}
 \sum_{\abs\alpha\le s-1}
 \left| -\int n_1\pa^\alpha\dot v\cdot\nabla q_\alpha
 +\frac12\frac{\dd}{\dd t}\int
 (Bq_\alpha^2+\eps^2\abs{\nabla q_\alpha}^2)\right|
 \le C\Gamma\cE_{s-1,\eps}^{\rm diff}.
 \label{eq:difference-electric-bound}
\end{equation}
In particular, no $\eps^{-1}$ factor has been introduced.

Finally commute \eqref{eq:difference-fluid-mass}--\eqref{eq:difference-fluid-momentum}, multiply them by
$a_1\pa^\alpha\dot r$ and $n_1\pa^\alpha\dot v$, and integrate.  The
principal acoustic terms cancel because $n_1a_1=p_i'(n_1)$.  Derivatives of
the weights are bounded by $C\Gamma$ using the equations for the first
solution.  Every source on the right of
\eqref{eq:difference-fluid-mass}--\eqref{eq:difference-fluid-momentum}
contains $\dot r$ or $\dot v$ times a profile derivative or a derivative of
the second solution, and hence contributes at most
$C\Gamma\cE_{s-1,\eps}^{\rm diff}$.  Adding
\eqref{eq:difference-electric-bound} proves \eqref{eq:difference}.  Gronwall
gives uniqueness and local Lipschitz dependence in $\cG^{s-1}_\eps$.
\end{proof}

\begin{proof}[Proof of \Cref{thm:main}]
Let $C_S$ be the norm of $H^s(\Om)\hookrightarrow L^\infty(\Om)$.  First
choose $\eta>0$ so that
\begin{equation}
 2\eta<\eta_P,\qquad 2C_S\eta\le\frac{n_*}{8},
 \label{eq:bootstrap-smallness}
\end{equation}
and so that \Cref{lem:energy-equivalence} holds throughout the $2\eta$
ball.  Reduce $\eps_0$ so that \Cref{prop:composite} gives
\begin{equation}
 \frac{3n_*}{4}\le n^a(t,x)\le\frac{5n^*}{4}
 \quad(t_0\le t\le T).
 \label{eq:profile-density-margin}
\end{equation}
The initial state-range hypothesis and the positivity in
\eqref{eq:fluid-weight-bounds} imply
\begin{equation}
 \norm{r_0}_{H^s}+\norm{v_0}_{H^s}\le C\eta_0.
 \label{eq:initial-fluid-smallness}
\end{equation}
The compatible-potential estimate then gives
\begin{equation}
 \norm{\psi_0}_{H^s}+\eps\norm{\nabla\psi_0}_{H^s}
 \le C(\eta_0+\eps^{2M+2}).
 \label{eq:initial-graph-smallness}
\end{equation}
Thus, after choosing $\eta_0$ and $\eps_0$ smaller if necessary, the data
lie strictly inside the local Poisson chart and the bootstrap set.

Let $C_*$ be the larger constant in the integrated energy estimate and in
the energy equivalence.  Choose $\eta_0$ and then reduce $\eps_0$ so that
\begin{equation}
  C_*\bigl(\eta_0^2+C_{M,t_0,T}\eps_0^{4M+4}\bigr)
  \le\frac14c_0\eta^2. \label{eq:small-choice}
\end{equation}

Let $[t_0,T_\eps^{\max})$ be the maximal perturbative interval from
\Cref{prop:local-theory}.  Define $\tau_\eps$ as the supremum of times
$\tau\le\min\{T,T_\eps^{\max}\}$ for which
\begin{equation}
 \sup_{t_0\le t<\tau}\mathfrak N_s(t)<2\eta,\qquad
 n(t,x)\in(n_*/4,2n^*),\qquad
 \norm{r(t)}_{H^s}<\eta_P.
 \label{eq:bootstrap-exit-time}
\end{equation}
The initial margins show $\tau_\eps>t_0$.  On
$[t_0,\tau_\eps)$, \Cref{prop:energy} and
\eqref{eq:residual-bound} give
\begin{align}
  \cE_{s,\eps}(t)
  &\le C_*\left(\cE_{s,\eps}(t_0)
  +\int_{t_0}^t(\mathfrak R_s^a(\sigma))^2\,\dd\sigma\right)\notag\\
  &\le C_*\bigl(\eta_0^2+
  C_{M,t_0,T}\eps^{4M+4}\bigr)
  \le\frac14c_0\eta^2.
 \label{eq:bootstrap-energy-improvement}
\end{align}
Energy equivalence now yields
\begin{equation}
 \sup_{t_0\le t<\tau_\eps}\mathfrak N_s(t)\le\frac{\eta}{2}.
 \label{eq:bootstrap-norm-improvement}
\end{equation}
In particular $\norm r_{H^s}<\eta_P/4$ after a harmless further reduction of
$\eta$.  By \eqref{eq:profile-density-margin}, Sobolev embedding, and
\eqref{eq:bootstrap-norm-improvement},
\begin{equation}
 n(t,x)\in
 \left[\frac{3n_*}{4}-\frac{C_S\eta}{2},
       \frac{5n^*}{4}+\frac{C_S\eta}{2}\right]
 \Subset(n_*/4,2n^*).
 \label{eq:bootstrap-density-improvement}
\end{equation}
Thus every inequality defining \eqref{eq:bootstrap-exit-time} is improved
with a strict margin independent of $\eps$.

If $\tau_\eps<T$ but $\tau_\eps<T_\eps^{\max}$, continuity would preserve
these strict inequalities past $\tau_\eps$, contradicting its definition.
If instead $\tau_\eps=T_\eps^{\max}<T$, then
\eqref{eq:bootstrap-norm-improvement} and
\eqref{eq:bootstrap-density-improvement} exclude every alternative in
\eqref{eq:continuation-criterion}; the solution restarts, again a
contradiction.  Hence $\tau_\eps=T$ for every
$0<\eps\le\eps_0$.  Returning to
\eqref{eq:bootstrap-energy-improvement} proves \eqref{eq:main-estimate}.
The difference estimate proves uniqueness and continuous dependence on the
whole interval.
\end{proof}

\subsection{Quasineutral expansion and vorticity}\label{sec:limit}

\begin{proof}[Proof of \Cref{cor:expansion}]
Combine \eqref{eq:main-estimate}, \Cref{lem:energy-equivalence}, and the
prepared-data assumption to obtain \eqref{eq:expansion-estimate}.  The
composite definition gives
\[
  \norm{n^a-\Nr}_{H^s}
  +\norm{u^a-(\Vr,0)}_{H^s}
  +\norm{\phi^a-\log\Nr}_{H^s}\le C\eps^2.
\]
The triangle inequality proves \eqref{eq:leading-rate}.
\end{proof}

\begin{corollary}[Identification of the Debye corrector]
\label{cor:first-corrector}
Assume the preparation in \Cref{cor:expansion} with $M\ge1$.  Then
\begin{equation}
\begin{split}
 \sup_{t_0\le t\le T}\biggl(&
 \norm{\eps^{-2}(n-\Nr)-N_1}_{H^s}
 +\norm{\eps^{-2}(u-(\Vr,0))-U_1}_{H^s}\\
 &+\norm{\eps^{-2}(\phi-\log\Nr)-\Phi_1}_{H^s}\biggr)
 \le C\eps^2. \label{eq:first-corrector-limit}
\end{split}
\end{equation}
Thus the forced system \eqref{eq:first-corrector-mass}--
\eqref{eq:first-corrector-momentum}, together with
\eqref{eq:first-phi}, is the actual first variation of the
Euler--Poisson solution around the quasineutral smooth simple wave.
\end{corollary}

\begin{proof}
Subtract the first two terms of \eqref{eq:corrector-ansatz}.  The remaining
profile terms start at order $\eps^4$, while
\eqref{eq:expansion-estimate} is $O(\eps^{2M+2})=O(\eps^4)$ for $M\ge1$.
Divide by $\eps^2$.
\end{proof}

\begin{corollary}[Every truncated Debye expansion]
\label{cor:truncated-expansion}
Assume the $M$th-order preparation in \Cref{cor:expansion}.  For
$0\le L\le M$, set
\[
 n^{[L]}=\sum_{j=0}^L\eps^{2j}N_j,\qquad
 u^{[L]}=\sum_{j=0}^L\eps^{2j}U_j,\qquad
 \phi^{[L]}=\sum_{j=0}^L\eps^{2j}\Phi_j.
\]
Then
\begin{equation}
 \sup_{t_0\le t\le T}\bigl(
 \norm{n-n^{[L]}}_{H^s}+\norm{u-u^{[L]}}_{H^s}
 +\norm{\phi-\phi^{[L]}}_{H^s}
 +\eps\norm{\nabla(\phi-\phi^{[L]})}_{H^s}\bigr)
 \le C_{L,M,t_0,T}\eps^{2L+2}.
 \label{eq:every-truncation}
\end{equation}
Thus the construction is an asymptotic expansion in the usual nested
sense, not only an estimate for the full $M$th-order sum.
\end{corollary}

\begin{proof}
Write, for example,
\[
 n-n^{[L]}=(n-n^a)+\sum_{j=L+1}^M\eps^{2j}N_j.
\]
The first term is $O(\eps^{2M+2})$ by
\eqref{eq:expansion-estimate}.  Each fixed corrector in the finite sum is
uniformly bounded in the required Sobolev norm by
\Cref{prop:composite}, and $0<\eps\le1$ gives
$\sum_{j=L+1}^M\eps^{2j}\le C_M\eps^{2L+2}$.  The same argument applies to
the velocity and potential.  After one spatial derivative of the potential
tail, the additional energy weight $\eps$ only improves the order.  This
proves \eqref{eq:every-truncation}.
\end{proof}

The theorem permits rotational perturbations.  Indeed, let
\[
  \omega=\pa_1u^2-\pa_2u^1,
  \qquad \zeta=\frac{\omega}{n}.
\]

\begin{proposition}[Specific vorticity]\label{prop:vorticity}
Every classical solution of \eqref{eq:EP-mass}--\eqref{eq:EP-poisson}
satisfies
\begin{equation}
  (\pa_t+u\cdot\nabla)\zeta=0. \label{eq:specific-vorticity}
\end{equation}
Consequently,
\begin{equation}
  \sup_{t_0\le t\le T}\norm{\zeta(t)}_{H^{s-1}}
  \le C_{t_0,T}\norm{\zeta(t_0)}_{H^{s-1}}. \label{eq:vorticity-bound}
\end{equation}
The constant is independent of $\eps$.
\end{proposition}

\begin{proof}
Take the scalar curl of \eqref{eq:EP-momentum}.  Both
$\nabla h_i(n)$ and $\nabla\phi$ are gradients, so
\[
  (\pa_t+u\cdot\nabla)\omega+\omega\Div u=0.
\]
Combine this with
$(\pa_t+u\cdot\nabla)n=-n\Div u$ to obtain
\eqref{eq:specific-vorticity}.  Commuting spatial derivatives through the
transport equation and applying the usual transport estimate gives
\eqref{eq:vorticity-bound}, since $\nabla u\in L^1([t_0,T];W^{s-2,\infty})$
by \Cref{thm:main}.  The electric field does not enter the estimate.
\end{proof}

\section{Conclusion and Outlook}\label{sec:conclusion}

We have established finite-time, Debye-uniform nonlinear stability of a
smooth expanding Euler--Poisson--Boltzmann simple wave under genuinely
two-dimensional perturbations, together with an arbitrary finite-order
quasineutral expansion. The decisive estimate couples the warm-ion
symmetrizer to the time-differentiated screened constraint, successfully closing the
top electric work without an $\eps^{-1}$ loss. This mechanism is fully compatible
with distinct neutral end states and transported vorticity.

The reference constructed in \Cref{prop:smooth-simple-wave} constitutes a smooth exact
solution of the effective Euler system. Because the Burgers datum is smooth and
monotone, all coefficient and corrector norms required by the energy method are
finite on $[t_0,T]$ for every $t_0\ge0$. Consequently, the theorem yields a
genuine initial-time stability result for the smooth expanding wave. While it
approximates the centered rarefaction at the quantitative rate \eqref{eq:centered-comparison},
the distinct geometric challenge of the Riemann vertex requires a separate treatment.

Specifically, for a centered self-similar fan $U^c(x_1/t)$,
\[
  \pa_1^2\log N^c(t,x_1)=t^{-2}(\log N^c)''(x_1/t).
\]
This indicates that the first outer Poisson correction for the centered fan is of size
$(\eps/t)^2$, meaning the expansion ceases to be uniform in the critical region
$t\lesssim\eps$. Extending the present stability framework to a nonsmooth reference
reaching its vertex will therefore necessitate constructing an inner Euler--Poisson vertex profile
and matching it to the outer Debye hierarchy developed here.

Similarly, extending the present framework to a global-in-time theorem requires
transitioning from the integrable coefficient bounds used on the finite interval
\eqref{eq:differential-energy} to a signed and weighted rarefaction estimate.
The compensated energy topology established in this paper provides the foundational
singular-perturbation machinery required to prove that the highly oscillatory electric
commutators can be integrated into such a global weighted hierarchy.

\appendix

\section{Appendix}\label{sec:appendix}

\subsection{Tame calculus on the cylinder}\label{app:moser}

We collect the precise variants of the product estimates used in the
profile, energy, and difference arguments.  They are stated in a form that
allows a coefficient to have different limits as $x_1\to\pm\infty$.

\begin{lemma}[Products and commutators]\label{lem:appendix-tame}
Let $m\ge3$, let $g\in H^m(\Om)$, and suppose
$f\in L^\infty(\Om)$ with $\nabla f\in H^{m-1}(\Om)$.  Then
\begin{align}
 \norm{fg}_{H^m}
 &\le C_m\left(\norm f_\infty\norm g_{H^m}
 +\norm{\nabla f}_{H^{m-1}}\norm g_\infty\right),
 \label{eq:relative-product}\\
 \sum_{1\le\abs\alpha\le m}
 \norm{[\pa^\alpha,f]g}_2
 &\le C_m\left(\norm{\nabla f}_\infty\norm g_{H^{m-1}}
 +\norm{\nabla f}_{H^{m-1}}\norm g_\infty\right),
 \label{eq:relative-commutator}\\
 \sum_{1\le\abs\alpha\le m}
 \norm{[\pa^\alpha,f]\nabla g}_2
 &\le C_m\left(\norm{\nabla f}_\infty\norm g_{H^m}
 +\norm{\nabla f}_{H^{m-1}}\norm{\nabla g}_\infty\right).
 \label{eq:transport-commutator-appendix}
\end{align}
If both $f$ and $g$ belong to $H^m$, the first estimate reduces to
\begin{equation}
 \norm{fg}_{H^m}\le C_m\left(
 \norm f_\infty\norm g_{H^m}+\norm g_\infty\norm f_{H^m}\right).
 \label{eq:moser-product}
\end{equation}
\end{lemma}

\begin{proof}
For $\abs\alpha\le m$, Leibniz' rule reads
\begin{equation}
 \pa^\alpha(fg)=f\pa^\alpha g+
 \sum_{0<\beta\le\alpha}\binom{\alpha}{\beta}
 \pa^\beta f\,\pa^{\alpha-\beta}g.
 \label{eq:appendix-Leibniz}
\end{equation}
The first term is bounded by $\norm f_\infty\norm g_{H^m}$.
If $1\le\abs\beta\le m-2$, then
$\pa^\beta f\in H^{m-\abs\beta}$ and the latter exponent is at least two;
the two-dimensional embedding $H^2\hookrightarrow L^\infty$ puts this
factor in $L^\infty$ and the other one in $L^2$.  If
$\abs\beta\ge m-1$, then $\abs{\alpha-\beta}\le1$; put the derivative of
$g$ in $L^\infty$ using $H^m\hookrightarrow W^{1,\infty}$ and the
derivative of $f$ in $L^2$.  Summing the finitely many splittings proves
\eqref{eq:relative-product}.

In $[\pa^\alpha,f]g$, the term with $\beta=0$ in
\eqref{eq:appendix-Leibniz} is absent.  The preceding low--high division
therefore gives \eqref{eq:relative-commutator}, with only derivatives of
$f$ on the right.  Replacing $g$ by $\nabla g$ gives
\eqref{eq:transport-commutator-appendix}; the apparent derivative of order
$m+1$ cancels because the $\beta=0$ term is again absent.  The estimate
\eqref{eq:moser-product} follows by using
$\norm{\nabla f}_{H^{m-1}}\le\norm f_{H^m}$ and its symmetric counterpart.
\end{proof}

\begin{lemma}[Composition and differences]\label{lem:appendix-composition}
Let $I\Subset\R$ be compact and $F\in C^{m+1}(I)$.
If $f$ takes values in $I$, $f\in L^\infty$, and
$\nabla f\in H^{m-1}$, then
\begin{equation}
 \norm{\nabla F(f)}_{H^{m-1}}
 \le C_{F,I,m}\left(1+\norm{\nabla f}_{H^{m-1}}^m\right).
 \label{eq:relative-composition}
\end{equation}
If $f_1,f_2$ take values in $I$, $f_1-f_2\in H^m$, and their relative
coefficient norms are bounded by $R$, then
\begin{equation}
 \norm{F(f_1)-F(f_2)}_{H^m}
 \le C_{F,I,m,R}\norm{f_1-f_2}_{H^m}.
 \label{eq:composition-difference}
\end{equation}
For a function $f\in H^m$ and a normalization $F(0)=0$, this also gives
\begin{equation}
 \norm{F(f)}_{H^m}
 \le C_{F,\norm f_\infty,m}
 (1+\norm f_{H^m}^{m-1})\norm f_{H^m}.
 \label{eq:moser-composition}
\end{equation}
\end{lemma}

\begin{proof}
Every differentiated composition is a finite sum of terms
\begin{equation}
 F^{(\ell)}(f)
 \pa^{\beta_1}f\cdots\pa^{\beta_\ell}f,
 \qquad
 \abs{\beta_j}\ge1,\qquad
 \beta_1+\cdots+\beta_\ell=\alpha.
 \label{eq:Faa-di-Bruno}
\end{equation}
The smooth coefficient is bounded in $L^\infty$ on $I$.  Put a factor with
the largest derivative count in $L^2$ and estimate the remaining factors
successively by the product argument in
\eqref{eq:appendix-Leibniz}; in dimension two the two-derivative reserve is
sufficient.  This proves \eqref{eq:relative-composition}.

For the difference estimate, use the exact identity
\begin{equation}
 F(f_1)-F(f_2)=(f_1-f_2)
 \int_0^1F'\bigl(f_2+\theta(f_1-f_2)\bigr)\,\dd\theta.
 \label{eq:composition-mean-value}
\end{equation}
The integral coefficient has bounded $L^\infty$ norm and bounded relative
derivatives through order $m$, by
\eqref{eq:relative-composition}.  Applying
\eqref{eq:relative-product} proves
\eqref{eq:composition-difference}.  Taking $f_2=0$ gives
\eqref{eq:moser-composition}.
\end{proof}

\subsection{Fourier cutoffs and passage to the limit}\label{app:cutoffs}

\begin{lemma}[Cutoff estimates]\label{lem:appendix-cutoff}
Let $J_\nu$ be the cutoff used in
\eqref{eq:Friedrichs-mass}--\eqref{eq:Friedrichs-momentum}.  If
$0\le r\le q$, then
\begin{equation}
 \norm{(I-J_\nu)f}_{H^r}
 \le C\nu^{r-q}\norm f_{H^q},\qquad
 \norm{J_\nu f}_{H^q}\le C\norm f_{H^q}.
 \label{eq:cutoff-quantitative}
\end{equation}
In particular, for a family bounded in $H^s$,
\[
 \norm{(J_\nu-J_\mu)f}_{H^{s-1}}
 \le C(\nu^{-1}+\mu^{-1})\norm f_{H^s}.
\]
For a family bounded only in $H^{s-1}$, the estimate actually used on the
nonlinear hyperbolic right-hand side is
\begin{equation}
 \norm{(J_\nu-J_\mu)f}_{H^{s-2}}
 \le C(\nu^{-1}+\mu^{-1})\norm f_{H^{s-1}}.
 \label{eq:cutoff-right-hand-side}
\end{equation}
Moreover, for a coordinate derivative $\partial_k$ and an integer $m\ge0$,
\begin{equation}
 \norm{[J_\nu,a]\partial_k f}_{H^m}
 \le C_m\left(\norm{\nabla a}_{L^\infty}\norm f_{H^m}
 +\norm{\nabla a}_{H^m}\norm f_{L^\infty}\right),
 \label{eq:cutoff-commutator}
\end{equation}
with a constant independent of $\nu$ whenever the right-hand side is
finite.
\end{lemma}

\begin{proof}
Expand in Fourier transform in $x_1$ and Fourier series in $x_2$.  On the
support of $1-\chi(\xi/\nu)$, the full frequency
$\langle\xi\rangle$ is bounded below by a fixed multiple of $\nu$.
Therefore
\[
 \langle\xi\rangle^{2r}|1-\chi(\xi/\nu)|^2
 \le C\nu^{2(r-q)}\langle\xi\rangle^{2q}.
\]
Integration and summation give the first estimate; boundedness of the symbol
gives the second.  The difference estimate follows from
$J_\nu-J_\mu=(I-J_\mu)-(I-J_\nu)$.

For \eqref{eq:cutoff-commutator}, write $J_\nu$ as convolution with
$K_\nu(x)=\nu^2K(\nu x)$ on the cylinder, with the periodic variable
understood by periodization.  In the order-zero case, integrate by parts in
the identity
\[
 [J_\nu,a]\partial_kf(x)
 =\int K_\nu(x-y)(a(y)-a(x))\partial_kf(y)\,\dd y.
\]
The first moment bound for $K_\nu$ cancels the factor produced by
$\nabla K_\nu$, giving
$\norm{[J_\nu,a]\partial_kf}_2
 \le C\norm{\nabla a}_\infty\norm f_2$ uniformly in $\nu$.
Apply $\pa^\alpha$, expand by Leibniz' rule, use this order-zero estimate on
the term with all derivatives on $f$, and use
\eqref{eq:relative-product} on the remaining terms.  This proves
\eqref{eq:cutoff-commutator}.
\end{proof}

The compactness passage in \Cref{prop:local-theory} uses only the following
consequence.  If $z_\nu$ is bounded in $L^\infty H^s$, Cauchy in
$C H^{s-2}$, and $\pa_tz_\nu$ is bounded in $L^\infty H^{s-1}$, then
interpolation gives strong convergence in $C H^{s'}$ for every $s'<s$;
in particular it recovers $C H^{s-1}$.  The limit is weakly continuous in
$H^s$, and products pass because one may choose $s'>2$.  Applying a spatial
mollifier directly to the limit equation, passing its commutators by the
tame estimates, and then letting the mollification scale tend to zero gives
the top-index integral energy identity.  Continuity of that energy and weak
$H^s$ continuity imply strong $H^s$ continuity as in
\eqref{eq:top-strong-continuity}. This convergence architecture succeeds strictly 
through frequency interpolation and energy continuity, circumventing any reliance 
on compact embedding in the unbounded $x_1$ direction.

\subsection{Constant dependence}\label{app:constants}

All constants in the profile construction may depend on
\[
  s,\ M,\ t_0,\ T,\ n_*,\ n^*,\ p_*,\
  \norm{p_i}_{C^{s+3M+6}(I^\sharp)},\ \delta_0,\text{ and }w_0,
\]
but not on $\eps$.  The stability constants also depend on the fixed
rarefaction family and the period of $\T$.  Smallness is chosen in the order
\[
  \text{state interval and pressure}\;\longrightarrow\;
  \text{profile bound}\;\longrightarrow\;
  \eta\;\longrightarrow\;\eta_0\;\longrightarrow\;\eps_0.
\]
This strict selection hierarchy prevents any implicit $\eps$ dependence from contaminating the lifespan.

\bibliography{1references}

@article{cordier2000quasineutral,
  title={Quasineutral limit of an Euler-Poisson system arising from plasma physics},
  author={Cordier, St{\'e}phane and Grenier, Emmanuel},
  journal={Communications in Partial Differential Equations},
  volume={25},
  number={5-6},
  pages={1099--1113},
  year={2000},
  publisher={Taylor \& Francis}
}

@article{wang2005quasineutral,
  title={Quasineutral limit of Euler--Poisson system with and without viscosity},
  author={Wang, Shu},
  journal={Communications in Partial Differential Equations},
  volume={29},
  number={3-4},
  pages={419--456},
  year={2005},
  publisher={Taylor \& Francis}
}

@article{peng2006quasi,
  title={Quasi-neutral limit of the non-isentropic Euler--Poisson system},
  author={Peng, Yue-Jun and Wang, Ya-Guang and Yong, Wen-An},
  journal={Proceedings of the Royal Society of Edinburgh Section A: Mathematics},
  volume={136},
  number={5},
  pages={1013--1026},
  year={2006},
  publisher={Royal Society of Edinburgh Scotland Foundation}
}

@article{slemrod2001quasi,
  title={Quasi-neutral limit for Euler-Poisson system},
  author={Slemrod, Marshall and Sternberg, Natalia},
  journal={Journal of Nonlinear Science},
  volume={11},
  number={3},
  pages={193--209},
  year={2001},
  publisher={Springer}
}

@article{gerard2013quasineutral,
  title={Quasineutral limit of the Euler-Poisson system for ions in a domain with boundaries},
  author={G{\'e}rard-Varet, David and Han-Kwan, Daniel and Rousset, Fr{\'e}d{\'e}ric},
  journal={Indiana University Mathematics Journal},
  pages={359--402},
  year={2013},
  publisher={JSTOR}
}

@article{gerard2014quasineutral,
  title={Quasineutral limit of the Euler-Poisson system for ions in a domain with boundaries II},
  author={G{\'e}rard-Varet, David and Han-Kwan, Daniel and Rousset, Fr{\'e}d{\'e}ric},
  journal={Journal de l’{\'E}cole polytechnique—Math{\'e}matiques},
  volume={1},
  pages={343--386},
  year={2014}
}

@article{luo2025stability1,
  title={On the stability of multi-dimensional rarefaction waves I: the energy estimates},
  author={Luo, Tian-Wen and Yu, Pin},
  journal={Annals of Mathematics},
  volume={202},
  number={2},
  pages={631--752},
  year={2025},
  publisher={Department of Mathematics, Princeton University Princeton, New Jersey, USA}
}

@article{luo2025stability2,
  title={On the stability of multi-dimensional rarefaction waves II: existence of solutions and applications to the Riemann problem},
  author={Luo, Tian-Wen and Yu, Pin},
  journal={Annals of Mathematics},
  volume={202},
  number={2},
  pages={753--855},
  year={2025},
  publisher={Department of Mathematics, Princeton University Princeton, New Jersey, USA}
}

@article{he2026extra,
  title={The Extra Vanishing Structure and Nonlinear Stability of Multi-Dimensional Rarefaction Waves: The Geometric Weighted Energy Estimates},
  author={He, Haoran and He, Qichen},
  journal={arXiv preprint arXiv:2603.05332},
  year={2026}
}

@article{jia2026multi,
  title={Multi-Dimensional Structural Stability of Mixed Riemann Configurations Containing Centered Rarefaction Waves and Surfaces of Discontinuities of Gas Dynamics},
  author={Jia, Jin and Luo, Tao},
  journal={arXiv preprint arXiv:2603.14696},
  year={2026}
}

@article{duan2015stability1,
  title={Stability of rarefaction waves of the Navier--Stokes--Poisson system},
  author={Duan, Renjun and Liu, Shuangqian},
  journal={Journal of Differential Equations},
  volume={258},
  number={7},
  pages={2495--2530},
  year={2015},
  publisher={Elsevier}
}

@article{duan2015stability2,
  title={Stability of the Rarefaction Wave of the Vlasov--Poisson--Boltzmann System},
  author={Duan, Renjun and Liu, Shuangqian},
  journal={SIAM Journal on Mathematical Analysis},
  volume={47},
  number={5},
  pages={3585--3647},
  year={2015},
  publisher={SIAM}
}

@article{degond2012numerical,
  title={Numerical approximation of the Euler-Poisson-Boltzmann model in the quasineutral limit},
  author={Degond, Pierre and Liu, Hailiang and Savelief, Dominique and Vignal, M-H},
  journal={Journal of Scientific Computing},
  volume={51},
  number={1},
  pages={59--86},
  year={2012},
  publisher={Springer}
}

@article{arun2025asymptotic,
  title={An asymptotic preserving scheme for the Euler-Poisson-Boltzmann system in the quasineutral limit},
  author={Arun, KR and Ghorai, Rahuldev},
  journal={Computers \& Mathematics with Applications},
  volume={185},
  pages={1--28},
  year={2025},
  publisher={Elsevier}
}

@article{ju2010quasi1,
  title={Quasi-neutral limit of the two-fluid Euler-Poisson system},
  author={Ju, Qiangchang and Li, Hailiang and Li, Yong and Jiang, Song},
  journal={Communications on Pure and Applied Analysis},
  volume={9},
  number={6},
  pages={1577--1590},
  year={2010},
  publisher={Communications on Pure and Applied Analysis}
}

@article{jiang2010quasi2,
  title={Quasi-neutral limit of the full bipolar Euler-Poisson system},
  author={Jiang, Song and Ju, QiangChang and Li, HaiLiang and Li, Yong},
  journal={Science China Mathematics},
  volume={53},
  number={12},
  pages={3099--3114},
  year={2010},
  publisher={Springer}
}

@article{alves2024zero,
  title={Zero-electron-mass and quasi-neutral limits for bipolar Euler--Poisson systems: NJ Alves and AE Tzavaras},
  author={Alves, Nuno J and Tzavaras, Athanasios E},
  journal={Zeitschrift f{\"u}r angewandte Mathematik und Physik},
  volume={75},
  number={1},
  pages={17},
  year={2024},
  publisher={Springer}
}

@article{pu2014quasineutral1,
  title={Quasineutral limit of the pressureless Euler--Poisson equation},
  author={Pu, Xueke},
  journal={Applied Mathematics Letters},
  volume={30},
  pages={33--37},
  year={2014},
  publisher={Elsevier}
}

@article{pu2016quasineutral2,
  title={Quasineutral limit of the pressureless Euler-Poisson equation for ions},
  author={Pu, Xueke and Guo, Boling},
  journal={Quarterly of Applied Mathematics},
  volume={74},
  number={2},
  pages={245--273},
  year={2016},
  publisher={Brown University}
}

@article{pu2016quasineutral3,
  title={QUASINEUTRAL LIMIT OF THE EULER-POISSON SYSTEM UNDER STRONG MAGNETIC FIELDS.},
  author={Pu, Xueke},
  journal={Discrete \& Continuous Dynamical Systems-Series S},
  volume={9},
  number={6},
  pages={2095},
  year={2016}
}

@article{brenier2000convergence,
  title={Convergence of the Vlasov-Poisson system to the incompressible Euler equations},
  author={Brenier, Yann},
  journal={Communications in Partial Differential Equations},
  volume={25},
  number={3-4},
  pages={737--754},
  year={2000},
  publisher={Taylor \& Francis}
}

@article{han2011quasineutral,
  title={Quasineutral limit of the Vlasov-Poisson system with massless electrons},
  author={Han-Kwan, Daniel},
  journal={Communications in Partial Differential Equations},
  volume={36},
  number={8},
  pages={1385--1425},
  year={2011},
  publisher={Taylor \& Francis}
}

@article{han2014quasineutral,
  title={The quasineutral limit of the Vlasov-Poisson equation in Wasserstein metric},
  author={Han-Kwan, Daniel and Iacobelli, Mikaela},
  journal = {Communications in Mathematical Sciences},
  volume  = {15},
  number  = {2},
  pages   = {481--509},
  year    = {2017}
}

@article{han2017quasineutral,
  title={Quasineutral limit for Vlasov--Poisson via Wasserstein stability estimates in higher dimension},
  author={Han-Kwan, Daniel and Iacobelli, Mikaela},
  journal={Journal of Differential Equations},
  volume={263},
  number={1},
  pages={1--25},
  year={2017},
  publisher={Elsevier}
}

@article{crispel2007asymptotic,
  title={An asymptotic preserving scheme for the two-fluid Euler--Poisson model in the quasineutral limit},
  author={Crispel, Pierre and Degond, Pierre and Vignal, Marie-H{\'e}l{\`e}ne},
  journal={Journal of Computational Physics},
  volume={223},
  number={1},
  pages={208--234},
  year={2007},
  publisher={Elsevier}
}

@article{degond2008analysis,
  title={Analysis of an asymptotic preserving scheme for the Euler--Poisson system in the quasineutral limit},
  author={Degond, Pierre and Liu, Jian-Guo and Vignal, Marie-H{\'e}l{\`e}ne},
  journal={SIAM Journal on Numerical Analysis},
  volume={46},
  number={3},
  pages={1298--1322},
  year={2008},
  publisher={SIAM}
}

@article{vignal2010boundary,
  title={A boundary layer problem for an asymptotic preserving scheme in the quasi-neutral limit for the Euler--Poisson system},
  author={Vignal, Marie H{\'e}l{\`e}ne},
  journal={SIAM Journal on Applied Mathematics},
  volume={70},
  number={6},
  pages={1761--1787},
  year={2010},
  publisher={SIAM}
}

@article{alinhac1989existence,
  title={Existence d'ondes de rarefaction pour des systems quasi-lineaires hyperboliques multidimensionnels},
  author={Alinhac, Serge},
  journal={Communications in partial differential equations},
  volume={14},
  number={2},
  pages={173--230},
  year={1989},
  publisher={Taylor \& Francis}
}

@article{matsumura1986asymptotics,
  title={Asymptotics toward the rarefaction waves of the solutions of a one-dimensional model system for compressible viscous gas},
  author={Matsumura, Akitaka and Nishihara, Kenji},
  journal={Japan Journal of Applied Mathematics},
  volume={3},
  number={1},
  pages={1--13},
  year={1986},
  publisher={Springer}
}

@article{liu1988nonlinear,
  title={Nonlinear stability of rarefaction waves for compressible Navier Stokes equations},
  author={Liu, Tai-Ping and Xin, Zhouping},
  journal={Communications in mathematical physics},
  volume={118},
  number={3},
  pages={451--465},
  year={1988},
  publisher={Springer}
}

@article{xin1993zero,
  title={Zero dissipation limit to rarefaction waves for the one-dimensional Navier-Stokes equations of compressible isentropic gases},
  author={Xin, Zhouping},
  journal={Communications on pure and applied mathematics},
  volume={46},
  number={5},
  pages={621--665},
  year={1993},
  publisher={Wiley Online Library}
}

\end{document}